\documentclass{amsart}

\usepackage[margin=1in]{geometry}

\usepackage{mathpazo}
\usepackage[T1]{fontenc}

\usepackage{graphicx}
\usepackage{amsmath}
\usepackage{amsthm, amssymb, amsfonts}
\usepackage{bm, bbm, mathtools, mathdots, mathrsfs}

\usepackage[shortlabels]{enumitem}
\setlist[enumerate]{label=(\alph*)}

\usepackage{rotating}
\usepackage{caption}
\usepackage[normalsize]{subfigure}

\usepackage{tikz}
\usetikzlibrary{cd,decorations.markings,arrows.meta,calc}

\usepackage[backref,hidelinks]{hyperref}
\usepackage[capitalize]{cleveref}

\newtheorem{theorem}{Theorem}[section]
\newtheorem{proposition}[theorem]{Proposition}
\newtheorem{lemma}[theorem]{Lemma}
\newtheorem{corollary}[theorem]{Corollary}
\newtheorem{conjecture}[theorem]{Conjecture}
\crefname{conjecture}{Conjecture}{Conjectures}
\theoremstyle{definition}

\newtheorem{remark}[theorem]{Remark}
\newtheorem{example}[theorem]{Example}
\newtheorem{question}[theorem]{Question}

\newcommand*{\RR}{\mathbb{R}}

\newcommand*{\ZZ}{\mathbb{Z}}

\newcommand*{\sM}{\mathscr{M}}

\newcommand*{\sP}{\mathscr{P}}

\newcommand*{\card}{\#}

\makeatletter
\newcommand*{\defeq}{\mathrel{\rlap{\raisebox{0.3ex}{$\m@th\cdot$}}\raisebox{-0.3ex}{$\m@th\cdot$}}=}
\makeatother
\newcommand*{\given}{\mid}
\newcommand*{\maps}{\nobreak\mskip2mu\mathpunct{}\nonscript
  \mkern-\thinmuskip{:}\mskip6muplus1mu\relax}

\newcommand*{\abs}[1]{{\lvert #1 \rvert}}

\DeclarePairedDelimiter{\floor}{\lfloor}{\rfloor}
\DeclarePairedDelimiterX\set[1]\lbrace\rbrace{\def\given{\;\delimsize\vert\;}#1}
\DeclarePairedDelimiterX\genset[1]\langle\rangle{\def\given{\;\delimsize\vert\;}#1}

\newcommand*{\rest}[2]{{
  \left.\kern-\nulldelimiterspace 
  #1 
  \vphantom{\big|} 
  \right|_{#2} 
  }}

\DeclareMathOperator{\im}{im}
\DeclareMathOperator{\Hom}{Hom}

\DeclareMathOperator{\Aut}{Aut}

\DeclareMathOperator{\supp}{supp}

\DeclareMathOperator{\cyc}{cyc}
\DeclareMathOperator{\val}{val}

\DeclareMathOperator{\Cer}{Cer}
\DeclareMathOperator{\Jac}{Jac}
\newcommand*{\id}{\mathrm{id}}
\newcommand*{\ev}{\mathrm{ev}}

\title{The Ceresa period from tropical homology}
\author{Caelan Ritter}
\date{\today}

\begin{document}
\maketitle

\begin{abstract}
Given a finite graph $G$, we define the Ceresa period $\alpha(G)$ as a tool for studying algebraic triviality of the tropical Ceresa cycle introduced by Zharkov.  We show that $\alpha(G) = 0$ if and only if $G$ is of hyperelliptic type; then a theorem of Corey implies that having $\alpha(G) = 0$ is a minor-closed condition with forbidden minors $K_4$ and $L_3$.
\end{abstract}


\section*{Introduction}

\renewcommand*{\thetheorem}{\Alph{theorem}}

Let $C$ be a smooth algebraic curve.  The Ceresa cycle of $C$ is a canonical algebraic $1$-cycle in the Jacobian that is homologically trivial.  A celebrated result of Ceresa in \cite{ceresa1983c} says that this cycle is nonetheless algebraically nontrivial for very general curves of genus at least 3.  It is known to be trivial for hyperelliptic curves, and understanding whether the converse holds is an active area of study (see, e.g., \cite{beauville2021non,bisogno2023group}).

In \cite{zharkov2015c}, Zharkov defines the tropical Ceresa cycle and a notion of algebraic equivalence for tropical cycles.  In analogy to Ceresa's original result, he proves that, for very general tropical curves with underlying graph $K_4$, the Ceresa cycle is algebraically nontrivial; in other words, there exists a countable union of hypersurfaces in the moduli space of tropical curves overlying $K_4$ away from which this property holds.  The main tool employed in the argument is a tropical homological invariant of tropical curves of genus 3 that vanishes whenever the Ceresa cycle is algebraically trivial.  In the present paper, we extend this tool, which we call the \textit{Ceresa period} $\alpha(C)$, to all tropical curves $C$.  This invariant lives in a quotient of the third exterior power of the universal cover of the tropical Jacobian.  Instead of working with a particular tropical curve $C$, we find it helpful to follow the approach of \cite{corey2022ceresa} and consider the underlying graph $G$ with variable edge lengths $x_e$.  We define a corresponding ``universal'' Ceresa period $\alpha(G)$ that tells us how $\alpha(C)$ behaves on the moduli space of tropical curves overlying $G$.

Unlike in the algebraic setting, the tropical Torelli map is not injective; one consequence of this is that there exist tropical curves that are not themselves hyperelliptic but whose Jacobians are isomorphic to those of hyperelliptic curves.  A graph $G$ is said to be of \textit{hyperelliptic type} if some (and, in fact, every) tropical curve with $G$ as an underlying graph satisfies this property.  Then our main result is that
\begin{theorem}\label{thm:A}
  $G$ has trivial Ceresa period if and only if $G$ is of hyperelliptic type.
\end{theorem}

\noindent It follows immediately from a result of Corey in \cite{corey2021tropical} that $\alpha(G) = 0$ if and only if it contains either of the graphs $K_4$ or $L_3$ as a minor (see \cref{fig:graphs}).  Moreover, by our \cref{prop:evaluation-converse}, the same holds for very general tropical curves.  Then \cref{prop:zharkov} implies that a very general tropical curve having either $K_4$ or $L_3$ as a graph minor has algebraically nontrivial Ceresa cycle.  In analogy to the classical question, we ask whether this ``very general'' hypothesis can be removed, i.e.,
\begin{question}
  Does every non-hyperelliptic-type tropical curve have algebraically nontrivial Ceresa cycle?
\end{question}

In \cite{corey2022ceresa}, the authors define a Ceresa--Zharkov class $\mathbf{w}_{\tau}(G)$ using the theory of mapping class groups and the Johnson homomorphism; the definition depends on a hyperelliptic involution $\tau$ of the genus-$g$ surface into which they embed $G$.  An immediate consequence of \cref{thm:A} and \cite[Theorem~5.11]{corey2022ceresa} is that a graph has trivial Ceresa period in our sense if and only if it has trivial Ceresa--Zharkov class.  In fact, in case $G$ is either $K_4$ or $L_3$, there is a choice of $\tau$ so that $\alpha(G) = \mathbf{w}_{\tau}(G)$ (after the appropriate identifications); compare \cref{ex:k4} with \cite[Proposition~5.7]{corey2022ceresa} and \cref{ex:l3} with \cite[Proposition~5.9]{corey2022ceresa}.  This generalizes an observation made in \cite[Remark~3.7]{corey2020ceresa}.  We believe that this should hold for all graphs:
\begin{conjecture}\label{conj:1}
  Let $G$ be a graph of genus $g$.  Then there exists a hyperelliptic involution $\tau$ of the surface of genus $g$ for which $\alpha(G) = \mathbf{w}_{\tau}(G)$.
\end{conjecture}
\noindent This result, if true, would be evidence of a close link between the tropical Ceresa cycle of \cite{zharkov2015c} and the tropical Ceresa class of \cite{corey2020ceresa}.

A useful tool in proving \cref{thm:A} is the explicit representative $\Theta(\Upsilon)$ for $\alpha(G)$ that we construct in \cref{sec:explicit-representative}.  It has coefficients that are homogeneous of degree 2 in the polynomial ring generated by edge length variables $x_e$ for each edge $e \in E(G)$.  By \cref{prop:alpha-edge-char}, there is an easy combinatorial condition on pairs of edges $(e,e')$ and triples of cycles $(\gamma_i,\gamma_j,\gamma_k)$ that tells us precisely when $\Theta(\Upsilon)$ contains a monomial $2 x_ex_{e'} \epsilon_i \wedge \epsilon_j \wedge \epsilon_k$ up to sign.  We expect this representative of $\alpha(G)$ to be helpful in the resolution of \cref{conj:1}.

\subsection*{Outline}  In \cref{sec:background}, we discuss the necessary background for tropical curves and their Jacobians, as well as for tropical homology and algebraic cycles on real tori with integral structures.  We finish the section by recalling Zharkov's definition of the tropical Ceresa cycle and introducing our invariant $\alpha(C)$.  In \cref{sec:definition}, we adapt these concepts so that they are well-defined as graph-theoretic invariants.  In particular, we define $\alpha(G)$ and show in \cref{sec:evaluation} that it specializes to $\alpha(C)$ when fixing edge lengths.  In \cref{sec:useful-tools}, we construct the special representative $\Theta(\Upsilon)$ as described above, proving in \cref{sec:basep-indep} that both $\alpha(G)$ and $\alpha(C)$ are independent of the choice of basepoint that goes into defining the Ceresa cycle.  Finally, in \cref{sec:forb-minor-char}, we use the tools that are developed in the previous sections to prove \cref{thm:A}.

\renewcommand*{\thetheorem}{\arabic{section}.\arabic{theorem}}

\subsection*{Acknowledgments}

The author thanks Farbod Shokrieh for his invaluable guidance throughout this project, Alexander Waugh for his helpful homology advice, and Samouil Molcho, Thibault Poiret, Felix R\"ohrle, and Jonathan Wise for various insightful conversations at the BIRS workshop on ``Curves: Algebraic, Tropical, and Logarithmic''.
\section{Background}\label{sec:background}

\subsection{Tropical curves}\label{sec:tropical-curves}

Let $G$ be a \textit{graph}, by which we mean a finite, connected multigraph with vertex and edge sets $V(G)$ and $E(G)$, respectively.  The \textit{valence} of a vertex $v$, denoted $\val(v)$, is the number of half-edges incident to $v$.  A \textit{leaf} is an edge incident to a vertex of valence one.  The \textit{genus} of $G$ is the quantity $\card E(G) - \card V(G) + 1$.

Fix an arbitrary orientation on the edges of $G$; then each edge $e$ has a head vertex $e^+$ and tail vertex $e^-$.  Given the additional data of a \textit{length function} $\ell \maps E(G) \to \RR_{>0}$, we construct a topological space 
\begin{equation*}
  C \defeq \left.\bigsqcup_{e \in E(G)} [0,\ell(e)] \right/{\sim},
\end{equation*}
where $[0,\ell(e)] \subset \RR$ is a closed interval of length $\ell(e)$ and the gluing relations on endpoints are given by the incidence relations in $G$, i.e., for all $e, f \in E(G)$ not necessarily distinct,
\begin{equation*}  
\begin{cases}
  (0,e) \sim (0,f) & \text{if } e^- = f^- \\
  (0,e) \sim (\ell(f),f) & \text{if } e^- = f^+ \\
  (\ell(e),e) \sim (\ell(f),f) & \text{if } e^+ = f^+.
\end{cases}
\end{equation*}
Observe that $C$ becomes a metric space with the ``shortest distance'' metric induced by edge lengths.  We say that a metric space $C$ obtained via this construction is a \textit{tropical curve}.  We further say that the edge-weighted graph $(G,\ell)$ is a \textit{model} for $C$, that $G$ \textit{underlies} $C$, and that $C$ \textit{overlies} $G$.  The \textit{genus} of $C$ is equivalently its first Betti number or the genus of any underlying graph.

We remark that our notion of tropical curve is what in the language of 1-dimensional rational polyhedral spaces should better be called a ``smooth'' tropical curve; see for instance \cite[Section~2.3]{gross2019tautological}.  Also, our definition departs slightly from that of \cite[Section~3.1]{mikhalkin2008tropical}, wherein leaves of a graph $G$ are prescribed to have infinite length.  This difference is of no critical importance: as we shall see in \cref{sec:contraction}, our main object of interest, the Ceresa period, is not affected by contracting all of the leaves.

\subsection{Real tori with integral structures}\label{sec:tropical-tori}

Given a ring $R$, an $R$-algebra $S$, and an $R$-module $M$, we write as a shorthand $M_S \defeq S \otimes_R M$ and abbreviate the primitive elements $s \otimes m$ of $M_S$ by $sm$.

Fix a free abelian group $N$ of rank $g$.  Then we may naturally regard $N$ as a (full-rank) lattice inside of the $g$-dimensional $\RR$-vector space $N_{\RR}$.  Fix another lattice $\Lambda \subset N_{\RR}$.  The quotient group $X \defeq N_{\RR}/\Lambda$ with the quotient topology and smooth structure induced from $N_{\RR}$ is a real torus of dimension $g$.  The tangent space at every point of $X$ is canonically isomorphic to the universal cover $N_{\RR}$, so the lattice $N$ defines what is known as an \textit{integral structure} on the real torus $X$.  We say that a tangent vector $v$ on $X$ is \textit{integral} if $v \in N$.

\subsection{Smooth tropical homology}\label{sec:smooth-trop-homol}

Recall that singular $q$-chains on a topological space $X$ are integral formal sums of continuous maps $\Delta^q \to X$, where $\Delta^q$ is the standard $q$-simplex 
\begin{equation*}
  \Delta^q \defeq \set*{(t_0,\ldots,t_q) \in \RR^{q+1} \given \sum_{i=0}^q t_i = 1,\, t_i \geq 0 \text{ for all $i$}}.
\end{equation*}
If $X$ is a smooth manifold, we may restrict our attention to \textit{smooth} simplices $\Delta^q \to X$; the boundary map on singular chains takes smooth $q$-chains to smooth $(q-1)$-chains, so we obtain the \textit{smooth singular chain complex} $S_{\bullet}(X)$.  We write $H_{q}(X)$ for the $q$-th homology of $S_{\bullet}(X)$.  It is well-known that the inclusion of $S_{\bullet}(X)$ into the usual singular chain complex is a chain homotopy equivalence, and therefore that $H_{q}(X)$ is isomorphic to the singular homology of $X$.

We now restrict to the case where $X = N_{\RR}/\Lambda$ is a real torus with integral structure $N$.  Say that a smooth $q$-simplex $\sigma$ in $X$ is \textit{affine} if it is obtained from an affine map $\RR^{q+1} \to N_{\RR}$ by restricting to $\Delta^q \subset \RR^{q+1}$ and pushing forward by the quotient $N_{\RR} \to X$.  Given points $u_i \in N_{\RR}$, we write $(u_0,\ldots,u_q)$ for the unique affine $q$-simplex defined by mapping the $i$-th vertex of $\Delta^q$ to $u_i$.  For any $\lambda \in \Lambda$, $(u_0+\lambda,\ldots,u_q+\lambda)$ and $(u_0,\ldots,u_q)$ represent the same affine simplex of the torus $X$.

The tropical homology of rational polyhedral spaces was introduced by \cite{itenberg2019tropical} and has been further studied in, e.g., \cite{jell2018lefschetz,jell2019superforms,gross2019sheaf}.  On a real torus $X$ with integral structure as above, a tropical $(p,q)$-simplex is a singular $q$-simplex $\sigma$ with the additional data of a $p$-fold wedge product of integral tangent vectors at some point in the image of $\sigma$; since we have canonically identified the tangent space at each point of $X$ with the universal cover, this extra data corresponds to an element of $N^{\wedge p}$, the $p$-th exterior power of $N$.  However, we would like to be able to integrate over tropical chains, so we modify the definition by replacing singular chains with smooth ones:
\begin{equation*}
S_{p,\bullet}(X) \defeq N^{\wedge p} \otimes_{\ZZ} S_{\bullet}(X).
\end{equation*}
We denote the $q$-th homology of this chain complex by $H_{p,q}(X)$.  By the universal coefficient theorem and the equivalence of smooth and singular homology noted above, $H_{p,q}(X)$ is isomorphic to the tropical homology in the sense of \cite{itenberg2019tropical}, and in fact,
\begin{equation*}
  H_{p,q}(X) \cong N^{\wedge p} \otimes_{\ZZ} H_{q}(X).
\end{equation*}
By the K\"unneth theorem, $H_{q}(X) \cong \Lambda^{\wedge q}$.  Fixing bases $\epsilon_1,\ldots,\epsilon_g$ for $N$ and $\lambda_1,\ldots,\lambda_g$ for $\Lambda$, the Eilenberg--Zilber map for singular homology determines explicit generators
\begin{equation}\label{eq:36}
\begin{split}
  H_{1,1}(X) &= \ZZ\genset*{\epsilon_i \otimes (0,\lambda_j) \given i,j \in [g]} \\
  H_{1,2}(X) &= \ZZ\genset*{\epsilon_i \otimes \big((0,\lambda_j,\lambda_j+\lambda_k) - (0,\lambda_k,\lambda_j+\lambda_k)\big) \given i,j,k \in [g],\, j < k},
\end{split}
\end{equation}
where $[g] = \set{1,\ldots,g}$.

\subsection{Tropical algebraic cycles}\label{sec:trop-algebr-cycl}

A \textit{(tropical) algebraic $k$-cycle} in a rational polyhedral space $X$ is a balanced, weighted, rational polyhedral complex of pure dimension $k$, defined up to refinement.  For details, we refer to \cite[Section~2]{allermann2010first} and \cite[Section~2]{allermann2016rational}.  In the case where $X = N_{\RR}/\Lambda$ is a real torus with integral structure $N$, a polyhedral complex in $X$ is a stratified closed subset that locally lifts to a polyhedral complex in $N_{\RR}$.  We call it \textit{weighted} if every facet is assigned an integer weight, and \textit{rational} if the affine hull of every face has a basis that is integral in the sense of \cref{sec:tropical-tori}.  Algebraic $k$-cycles form an abelian group $Z_k(X)$, and a morphism of rational polyhedral spaces $f \maps X \to Y$ induces a pushforward homomorphism $f_{*} \maps Z_k(X) \to Z_k(Y)$.  In analogy with the classical situation, there is a notion of rational equivalence of cycles, which is explored in \cite[Section~3]{allermann2016rational}.  Likewise, there is a notion of algebraic equivalence defined by \cite{zharkov2015c} and explored further in \cite[Section~5.1]{gross2019tautological}.

There exists a group homomorphism $\cyc \maps Z_k(X) \to H_{k,k}(X)$, called the \textit{cycle class map}, from algebraic $k$-cycles to tropical $(k,k)$-homology classes.  We note that $\cyc$ commutes with pushforward maps.  When $X$ is a real torus with integral structure as above and $A \in Z_1(X)$, we may describe an explicit representative of the homology class $\cyc(A)$ as follows.  Each facet $\sigma$ of $A$ is a line segment with rational slope, so we may lift it to a line segment in the universal cover $N_{\RR}$ with endpoints $u_{\sigma}$ and $u'_{\sigma}$ and primitive tangent vector $n_{\sigma} \in N$, i.e., $n_{\sigma}$ is a positive scalar multiple of $u'_{\sigma} - u_{\sigma}$ such that $a n_{\sigma} \in N$ for precisely $a \in \ZZ$.  Then 
\begin{equation}\label{eq:7}
  \cyc(A) \defeq \sum_{\sigma} n_{\sigma} \otimes (u_{\sigma},u'_{\sigma}),
\end{equation}
where the summation is taken over all facets $\sigma$ of $A$.  The balancing condition at the vertices of $A$ translates to $\cyc(A)$ having trivial boundary as a tropical $(1,1)$-chain.  For the general case, see \cite[Section~4.C]{jell2018lefschetz} or \cite[Section~5]{gross2019sheaf}.

\subsection{Tropical Jacobians}\label{sec:tropical-jacobians}

The tropical Jacobian is a real torus with integral structure that we can canonically associate to any tropical curve.  The original definition given by \cite[Section~6.1]{mikhalkin2008tropical} mimics that of the Jacobian of a complex algebraic curve; here we present an equivalent definition that avoids mention of $1$-forms.

Let $C$ be a tropical curve of genus $g$.  Given a model $(G,\ell)$ of $C$, let $C_1(G,\ZZ)$ denote the oriented simplicial 1-chains on $G$.  The corresponding simplicial homology $H_1(G,\ZZ) \subset C_1(G,\ZZ)$ of $G$ is isomorphic to the singular homology $H_1(C,\ZZ)$ of $C$; in what follows, we conflate $H_1(G,\ZZ)$ and $H_1(C,\ZZ)$ without further remark.  Furthermore, by the universal coefficient theorem, we may identify $H_1(C,\RR)^{\vee}$ with $H^1(C,\RR)$.  Fixing an orientation on the edges of $G$, we define a symmetric, bilinear map $[\cdot,\cdot] \maps C_1(G,\ZZ) \times C_1(G,\ZZ) \to \RR$ called the \textit{length pairing} by
\begin{equation}\label{eq:30}
  [e, e'] =
  \begin{cases}
    \ell(e) &\text{ if } e = e'  \\
    0 &\text{ otherwise}
  \end{cases}.
\end{equation}
For any refinement $G'$ of $G$, this pairing descends to $C_1(G',\ZZ) \times C_1(G',\ZZ)$ in a way that is compatible with the restriction of either entry to $H_1(C,\ZZ)$.  Likewise, $[\cdot,\cdot]$ formally extends to allow coefficients in $\RR$.
By the $\otimes$--$\Hom$ adjunction, $[\cdot,\cdot]$ induces a homomorphism $\pi \maps C_1(G,\ZZ) \to H^1(C,\RR)$ via $e \mapsto [e,\cdot]$.  Consider the lattice 
\begin{equation*}
  \Lambda \defeq \pi(H_1(C,\ZZ)) \subset H^1(C,\RR);
\end{equation*}
its dual lattice is defined by 
\begin{equation*}
  N \defeq \set*{u \in H^1(C,\RR) \given u(\gamma) \in \ZZ \text{ for all } \gamma \in H_1(C,\ZZ)} \cong H^1(C,\ZZ).
\end{equation*}
The \textit{tropical Jacobian} of $C$ is the real torus
\begin{equation*}
  \Jac(C) \defeq H^1(C,\RR)/\Lambda \cong N_{\RR}/\Lambda
\end{equation*}
with integral structure given by $N$.

Fix a point $v \in C$.  Following \cite[Section~4]{baker2011metric}, we define the \textit{Abel--Jacobi map} $\Phi_v \maps C \to \Jac(C)$ based at $v$ by
\begin{equation*}
  \Phi_v(w) \defeq \pi(\delta) = [\delta,\cdot],
\end{equation*}
where $\delta$ is a path in $C_1(G,\ZZ)$ from $v$ to $w$ for some model $(G,\ell)$ of $C$ containing both $v$ and $w$ as vertices.  We claim that $\Phi_v(w)$ is well-defined.  Indeed, the definition does not depend on $\delta$, since any other path $\delta'$ from $v$ to $w$ satisfies $\delta' - \delta \in H_1(C,\ZZ)$, hence $\pi(\delta') \equiv \pi(\delta) \pmod \Lambda$.  Moreover, it is independent of the choice of model $(G,\ell)$ because the length pairing is preserved under refinement.

\subsection{Tropical Ceresa cycle}\label{sec:trop-ceresa-cycle}

We identify $C$ with its fundamental algebraic cycle in $Z_1(C)$, i.e., the unique 1-cycle with support equal to all of $C$ and with weight one on every edge.  We also write $[-1] \maps \Jac(C) \to \Jac(C)$ for multiplication by $-1$.  Fix $v \in C$; following \cite{zharkov2015c}, we define the \textit{tropical Ceresa cycle} based at $v$ by
\begin{equation*}
  \Phi_{v,*}C - [-1]_{*}\Phi_{v,*}C \in Z_1(\Jac(C)),
\end{equation*}
where $\Phi_{v,*}$ and $[-1]_{*}$ are the induced pushforwards on algebraic cycles mentioned in \cref{sec:trop-algebr-cycl}.  Applying the cycle class map to the Ceresa cycle yields a tropical $(1,1)$-cycle $\Cer_v(C)$.  Choosing a model $(G,\ell)$ of $C$ so that $v \in V(G)$, one can show using \cref{eq:7} that
\begin{equation}\label{eq:5}
  \Cer_v(C) = \sum_{e \in E(G)} \ell(e)^{-1}\pi(e) \otimes \big( (\pi(\delta_e),\pi(\delta_e+e)) + (-\pi(\delta_e),-\pi(\delta_e+e)) \big),
\end{equation}
where $\delta_e$ is a path in $G$ from $v$ to $e^-$.  Since $[-1]_*$ acts trivially on the generators of $H_{1,1}(\Jac(C))$ given by \cref{eq:36} (and in fact induces the identity on $H_{k,k}(\Jac(C))$ for all $k$), it follows that $\Cer_v(C)$ is trivial in homology.

\subsection{Integration map}\label{sec:integration-map}

Let $\epsilon_1,\ldots,\epsilon_g$ denote a basis for $N$ and $z_1,\ldots,z_g$ the associated coordinate functions on $N_{\RR}$; then the differential 1-forms $dz_i$ on $N_{\RR}$ descend to $\Jac(C)$.  Recall from \cref{sec:smooth-trop-homol} that $S_{1,2}(\Jac(C))$ is generated by elements of the form $n \otimes \sigma$, where $n = \sum_{i=1} n_i\epsilon_i \in N$ and $\sigma \maps \Delta^2 \to \Jac(C)$ is a smooth $2$-simplex.  Given a differential form $\omega$ on $\Jac(C)$, we define the integral of $\omega$ over $\sigma$ by
\begin{equation*}
  \int_{\sigma} \omega \defeq \int_{\Delta^2} \sigma^{*}\omega.
\end{equation*}
Let $\binom{[g]}{3}$ denote the collection of all subsets of $[g]$ of cardinality $3$.  Given $I \in \binom{[g]}{3}$, we write $I = \set{i_1,i_2,i_3}$ for $i_1 < i_2 < i_3$ and adopt the multi-index notation $\epsilon_I \defeq \epsilon_{i_1} \wedge \epsilon_{i_2} \wedge \epsilon_{i_3}$.  Mimicking the ``determinantal 2-form'' for $K_4$ introduced by \cite{zharkov2015c}, we define the \textit{integration map} $\Theta \maps S_{1,2}(\Jac(C)) \to N_{\RR}^{\wedge 3}$ via
\begin{equation}\label{eq:38}
  \Theta(n \otimes \sigma) \defeq 2\sum_{I \in \binom{[g]}{3}} \int_{\sigma} \left( n_{i_1} dz_{i_2} \wedge dz_{i_3} - n_{i_2} dz_{i_1} \wedge dz_{i_3} + n_{i_3} dz_{i_1} \wedge dz_{i_2} \right) \epsilon_I.
\end{equation}
In the case that $\sigma = (u_0,u_1,u_2)$ is affine (see \cref{sec:smooth-trop-homol}), \cref{eq:38} reduces to
\begin{equation*}
  \Theta(n \otimes (u_0,u_1,u_2)) = \sum_{I \in \binom{[g]}{3}} \det(M_{I,[3]}) \epsilon_I,
\end{equation*}
where $M$ is the $g \times 3$ matrix whose columns are the entries of $n$, $u_1-u_0$, and $u_2-u_0$ with respect to the basis $\epsilon_1,\ldots,\epsilon_g$.  We see that $\Theta$ is coordinate-independent at least on affine $(1,2)$-chains by rewriting this expression in terms of exterior products:
\begin{equation}\label{eq:4}
  \Theta(n \otimes (u_0,u_1,u_2)) = n \wedge (u_1-u_0) \wedge (u_2-u_0).
\end{equation}

The boundary map $\partial \maps S_{1,3}(\Jac(C)) \to S_{1,2}(\Jac(C))$ sends $n \otimes \sigma \mapsto n \otimes \partial \sigma$.  Then by Stokes' theorem, $\Theta(n \otimes \partial \sigma) = 0$, so $\Theta$ descends to a map on homology $H_{1,2}(\Jac(C)) \to N_{\RR}^{\wedge 3}$.  Let 
\begin{equation*}
\sP \defeq \Theta(H_{1,2}(\Jac(C))) \subset N_{\RR}^{\wedge 3}.
\end{equation*}
Fixing a basis $\lambda_1,\ldots,\lambda_g$ for $\Lambda$, we may apply \cref{eq:4} to the affine generating set of $H_{1,2}(\Jac(C))$ given by \cref{eq:36} to find that
\begin{equation}\label{eq:12}
  \sP = \ZZ\genset*{2\epsilon_i \wedge \lambda_j \wedge \lambda_k \given i,j,k \in [g],\, j < k}.
\end{equation}

\subsection{Ceresa period}\label{sec:integr-invar}

We saw in \cref{sec:trop-ceresa-cycle} that $\Cer_v(C)$ is trivial in homology, so we may choose some $\Sigma \in S_{1,2}(\Jac(C))$ for which $\partial\Sigma = \Cer_v(C)$.  Define the \textit{Ceresa period} $\alpha(C)$ of $C$ to be the image of $\Theta(\Sigma)$ in the quotient $N_{\RR}^{\wedge 3}/\sP$.
We claim that $\alpha(C)$ is well-defined.  Indeed, given another $\Sigma' \in S_{1,2}(\Jac(C))$ with the same boundary, we have that $\Sigma - \Sigma' \in H_{1,2}(\Jac(C))$, hence $\Theta(\Sigma) \equiv \Theta(\Sigma') \pmod{\sP}$.  As the notation suggests, $\alpha(C)$ is independent of the basepoint $v \in C$; see \cref{cor:4}.  The following result is a straightforward generalization of \cite[Lemma~5]{zharkov2015c}.

\begin{proposition}\label{prop:zharkov}
  If the tropical Ceresa cycle of $C$ is algebraically trivial, then $\alpha(C) = 0$.
\end{proposition}

\begin{proof}
  Given the data of an algebraic equivalence between $\Phi_{v,*}C$ and $[-1]_{*}\Phi_{v,*}C$, the algebraic cycles whose difference is the Ceresa cycle, one may construct by \cite[Lemma~4]{zharkov2015c} an affine $(1,2)$-chain $\Sigma \defeq \sum_i n_i \otimes \sigma_i$ in $\Jac(C)$ for which 
\begin{equation*}
  \partial\Sigma = \cyc\left(\Phi_{v,*}C\right) - \cyc\left([-1]_{*}\Phi_{v,*}C\right) = \Cer_v(C),
\end{equation*}
and that satisfies the property that $n_i$ lies in the affine hull of $\sigma_i(\Delta^2)$ for each $i$.  Then $\Theta(n_i \otimes \sigma_i) = 0$ by \cref{eq:4}, hence $\Theta(\Sigma) = 0$.
\end{proof}

\section{Ceresa period of a graph}\label{sec:definition}

We would like to work with graphs rather than particular tropical curves.  To that end, we define ``universal'' versions of the Jacobian, homology, the Ceresa cycle, and the Ceresa period using edge weights in a polynomial ring.  We conclude this section by showing that these universal objects specialize to their tropical counterparts when we fix specific edge lengths.

\subsection{Universal Jacobian}\label{sec:universal-jacobian}

Given a graph $G$ of genus $g$ with arbitrary edge orientations, we define a polynomial ring $R \defeq \ZZ[x_e \given e \in E(G)]$.  Let $N \defeq H^1(G,\ZZ)$, and consider the free $R$-module $N_R \cong H^1(G,R) \cong \Hom(H_1(G,\ZZ),R)$.  We define a pairing $[\cdot,\cdot] \maps C_1(G,\ZZ) \times C_1(G,\ZZ) \to R$ via
\begin{equation}\label{eq:31}
  [e, e'] =
  \begin{cases}
    x_e &\text{ if } e = e'  \\
    0 &\text{ otherwise}
  \end{cases},
\end{equation}
which induces a homomorphism $\pi \maps C_1(G,\ZZ) \to N_R$ by $e \mapsto [e,\cdot]$.  Then $\Lambda \defeq \pi(H_1(G,\ZZ))$ is a free $\ZZ$-submodule of $N_R$ of rank $g$.  Formally, we define the \textit{universal Jacobian} $\Jac(G)$ of $G$ to be the triple $(N_{\RR}/\Lambda, N_{\RR}, N)$, although we shall conflate $\Jac(G)$ with the quotient group $N_{\RR}/\Lambda$ whenever the other data are clear from context.

It should be noted that, although we use the symbol $\Jac(G)$, our universal Jacobian is distinct from the Jacobian of a finite graph, also known as the abelian sandpile group or critical group, which has been studied in various contexts in physics, arithmetic geometry, and graph theory.  For more on this subject, see, for instance, \cite{bak1988self,lorenzini1989arithmetical,dhar1990self,gabrielov1993abelian,bacher1997lattice,nagnibeda1997jacobian,baker2007riemann}.  

\subsection{Universal homology}\label{sec:universal-homology}

We define homology theories on $\Jac(G)$ as follows.  Let $C_q(\Jac(G))$ denote the free abelian group generated by $N_R^{q+1}/{\sim}$, equivalence classes of ordered $q$-simplices, where we identify 
\begin{equation*}
(u_0+\lambda,\ldots,u_q+\lambda) \sim (u_0,\ldots,u_q)
\end{equation*}
for all $\lambda \in \Lambda$.  This becomes a chain complex $C_{\bullet}(\Jac(G))$ via the usual boundary maps
\begin{equation*}
  \partial(u_0,\ldots,u_q) = \sum_{i=0}^q (-1)^i(u_0,\ldots,\hat u_i,\ldots,u_q),
\end{equation*}
where $\hat{\cdot}$ means that the corresponding entry is omitted.  Let $H_q(\Jac(G))$ denote the $q$-th homology of $C_{\bullet}(\Jac(G))$.  We further define $C_{p,\bullet}(\Jac(G)) \defeq N^{\wedge p} \otimes C_{\bullet}(\Jac(G))$, with corresponding $q$-th homology $H_{p,q}(\Jac(G))$.We call a $(p,q)$-chain \textit{degenerate} if every $q$-simplex that it contains has repeated entries.

The universal coefficient theorem yields 
\begin{equation}\label{UCT}
  H_{p,q}(\Jac(G)) \cong N^{\wedge p} \otimes_{\ZZ} H_q(\Jac(G)).
\end{equation}
Comparing the following technical result to \cref{eq:36}, this says that $\Jac(G)$ has the homology we expect of a $g$-dimensional real torus.

\begin{lemma}\label{homology-generators}\leavevmode
  Fix bases $\epsilon_1,\ldots,\epsilon_g$ of $N$ and $\lambda_1,\ldots,\lambda_g$ of $\Lambda$.  Then
\begin{enumerate}
\item $H_{1,1}(\Jac(G)) = \ZZ\genset*{\epsilon_i \otimes (0,\lambda_j) \given i,j \in [g]}$ and
\item $H_{1,2}(\Jac(G)) = \ZZ\genset*{\epsilon_i \otimes \left((0,\lambda_j,\lambda_j + \lambda_k) - (0,\lambda_k,\lambda_j+\lambda_k)\right) \given i,j,k \in [g],\, j < k}$.
\end{enumerate}
\end{lemma}

\begin{proof}
  By restriction of scalars, the $R$-module $N_R$ inherits the structure of a free $\ZZ$-module of infinite rank, which we denote by $M$.  Then $M_{\RR}$ is an infinite-dimensional $\RR$-vector space that naturally contains $M$ as a $\ZZ$-submodule.  We endow the $g$-dimensional subspace $\Lambda_{\RR}$ of $M_{\RR}$ with the Euclidean topology.  Choose a subspace $W$ complementary to $\Lambda_{\RR}$ and give it the Euclidean norm with respect to some basis $w_1,w_2,\ldots$.  Give $M_{\RR} = \Lambda_{\RR} \times W$ the product topology.  The action of $\Lambda$ on $\Lambda_{\RR}$ by translation extends to an action on $M_{\RR}$; define $J \defeq M_{\RR}/\Lambda$ with the quotient topology.  Equivalently, we may write $J = \Lambda_{\RR}/\Lambda \times W$.  It is a straightforward exercise to show that $W \cong \varinjlim_n \RR^n$ is contractible via a straight-line homotopy to $0$; therefore, the singular homology $H_q(J)$ of $J$ is isomorphic to that of the $g$-dimensional real torus $\Lambda_{\RR}/\Lambda$.  In particular, just as in \cref{sec:smooth-trop-homol}, one can show using the Eilenberg--Zilber map that $H_1(J)$ is freely generated by the affine simplices $(0,\lambda_i)$ in $M_{\RR}$ for each $i$, while $H_2(J)$ is freely generated by $(0,\lambda_i,\lambda_i+\lambda_j) - (0,\lambda_j,\lambda_i+\lambda_j)$ for $i < j$.

  Let $S_{\bullet}$ denote the singular chain complex functor.  There is a natural surjection $S_{\bullet}(M_{\RR}) \to S_{\bullet}(J)$ that identifies simplices in $M_{\RR}$ that are $\Lambda$-translates of each other, so we may refer to simplices of $J$ by their representatives in $M_{\RR}$.  There is an injective chain map $\Phi \maps C_{\bullet}(\Jac(G)) \to S_{\bullet}(J)$ sending the ordered $q$-simplex $(u_0,\ldots,u_q)$ to the affine singular simplex $(u_0,\ldots,u_q)$; we naturally identify $C_{\bullet}(\Jac(G)) \cong \im\Phi$.  It is clear that the given generators for $H_q(J)$ for $q \in \set{1,2}$ are in $\im\Phi$, so the induced map $H_q(\Jac(G)) \to H_q(J)$ is surjective.  We claim that it is in fact an isomorphism.

  It suffices to show that the injection $\Phi$ splits, since then it must induce an injection on homology.  Indeed, let $A_{\bullet}(J) \subset S_{\bullet}(J)$ denote the subcomplex generated by affine simplices; by construction, $\im\Phi \subset A_{\bullet}(J)$.  We define chain maps $F \maps S_{\bullet}(J) \to A_{\bullet}(J)$ sending a singular simplex to the affine simplex with the same vertices and $G \maps A_{\bullet}(J) \to \im\Phi$ sending $(u_0,\ldots,u_q) \mapsto (\floor{u_0},\ldots,\floor{u_q})$, where $\floor{\cdot}$ applies the usual floor function to each coordinate of the argument with respect to the basis $\lambda_1,\ldots,\lambda_g,w_1,w_2,\ldots$ of $M_{\RR}$.
\begin{equation*}
\begin{tikzcd}
  C_{\bullet}(\Jac(G)) \ar[d, "\cong"] \ar[rrd, "\Phi"] && \\
  \im\Phi \ar[r, hook] & A_{\bullet}(J) \ar[r, hook]\ar[l, bend left=30, "G"] & S_{\bullet}(J) \ar[l, bend left=30, "F"]  
\end{tikzcd}
\end{equation*}
The restriction of $GF$ to $\im\Phi$ is the identity, as desired.

Since $H_q(\Jac(G)) \to H_q(J)$ is an isomorphism for $q \in \set{1,2}$, the chosen generators of $H_q(J)$ pull back to corresponding generators for $H_q(\Jac(G))$.  The final result then follows from \cref{UCT}.
\end{proof}

Suppose that $G$ and $G'$ are graphs with universal Jacobians $\Jac(G) = N_R/\Lambda$ and $\Jac(G') = N'_{R'}/\Lambda'$, respectively.  Then any $\ZZ$-linear map $f \maps N_R \to N'_{R'}$ for which $f(\Lambda) \subset \Lambda'$ and $f(N) \subset N'$ induces a chain map $f_{*} \maps C_{p,\bullet}(\Jac(G)) \to C_{p,\bullet}(\Jac(G'))$ via
\begin{equation}\label{eq:23}
f_{*}(n_1 \wedge \ldots \wedge n_p \otimes (u_0,\ldots,u_q)) \defeq f(n_1) \wedge \ldots \wedge f(n_p) \otimes (f(u_0),\ldots,f(u_q)).
\end{equation}

\subsection{Universal Ceresa cycle}

Fix a basepoint $v \in V(G)$ and define the \textit{Abel--Jacobi map} $\Phi_v \maps V(G) \to \Jac(G)$ by sending $w \mapsto \pi(\delta)$, where $\delta$ is a path from $v$ to $w$.  Just as for the Abel--Jacobi map associated to a tropical curve, $\Phi_v$ is well-defined modulo $\Lambda = \pi(H_1(G,\ZZ))$.  By an abuse of notation, we also define $\Phi_v \maps E(G) \to C_{1,1}(\Jac(G))$ by
\begin{equation}\label{eq:32}
  \Phi_v(e) \defeq x_e^{-1}\pi(e) \otimes (\pi(\delta_e),\pi(\delta_e + e)),
\end{equation}
where $\delta_e$ is a path from $v$ to $e^-$.  Notice that $\pi(\delta_e)$ and $\pi(\delta_e + e)$ are representatives in $N_R$ of $\Phi_v(e^-)$ and $\Phi_v(e^+)$ respectively.

Taking \cref{eq:5} as inspiration, we define the \textit{(universal) Ceresa cycle} of $G$ based at $v$ by
\begin{equation}\label{ceresa-cycle}
  \Cer_v(G) \defeq \sum_{e\in E(G)} \Phi_v(e) - [-1]_{*}\left(\sum_{e\in E(G)} \Phi_v(e)\right),
\end{equation}
where the inversion map $[-1] \maps N_R \to N_R$ induces a pushforward as in \cref{eq:23}.
\begin{lemma}\label{lem:4}
  $\Cer_v(G)$ is independent of the orientation on the edges up to adding the boundary of a degenerate $(1,2)$-chain.
\end{lemma}
\begin{proof}
  Let $\bar e$ be the edge $e$ with the opposite orientation.  We observe that $\bar e = -e$ in $C_1(G,\ZZ)$ and $\bar{e}^- = e^+$, so we may take $\delta_{\bar{e}} = \delta_e + e$.  We also have $x_{\bar{e}} = x_e$.  Then
\begin{align*}
  \Phi_v(\bar e)
  &= x_{\bar{e}}^{-1}\pi(\bar{e}) \otimes (\pi(\delta_{\bar{e}}),\pi(\delta_{\bar{e}}+\bar{e}))\\
  &= -x_e^{-1}\pi(e) \otimes (\pi(\delta_e + e),\pi(\delta_e + e + \bar{e}))\\
  &= -x_e^{-1}\pi(e) \otimes (\pi(\delta_e + e),\pi(\delta_e))\\
  &= \Phi_v(e) - \partial\left(x_e^{-1}\pi(e) \otimes \big((\pi(\delta_e),\pi(\delta_e+e),\pi(\delta_e)) + (\pi(\delta_e),\pi(\delta_e),\pi(\delta_e))\big)\right).\qedhere
\end{align*}
\end{proof}

\begin{lemma}\label{lem:1}
  $\sum_{e\in E(G)} \Phi_v(e) \in H_{1,1}(\Jac(G))$.
\end{lemma}
\begin{proof}
  We compute
\begin{equation*}
  \partial \Phi_v(e) = x_e^{-1}\pi(e) \otimes ((\pi(\delta_e+e)) - (\pi(\delta_e))).
\end{equation*}
Then $\partial\left(\sum_{e\in E(G)}\Phi_v(e)\right)$ is supported on the subset $\Phi_v(V(G)) \subset \Jac(G)$, so it suffices to fix $w \in V(G)$ and show that the part of the boundary supported on $\Phi_v(w)$ vanishes.  The only edges that have a boundary component at $\Phi_v(w)$ are those that are incident to $w$.  We further restrict our attention to the non-loop edges, since if $e$ is a loop edge, $\pi(e) \in \Lambda$, hence $\partial \Phi_v(e) = 0$.

Let $\delta$ be a path from $v$ to $w$, and label the non-loop edges adjacent to $w$ by $e_1,\ldots,e_n$.  By \cref{lem:4}, up to adding a boundary, we may assume that each $e_i$ is oriented away from $w$.  We write $\Phi_v(e_i) = x_{e_i}^{-1}\pi(e_i) \otimes (\pi(\delta),\pi(\delta+e_i))$.  This contributes a boundary component of $-x_{e_i}^{-1}\pi(e_i) \otimes (\pi(\delta))$ at $\Phi_v(w)$, so we need only show that $\sum_{i=1}^n x_{e_i}^{-1}\pi(e_i) = 0$, or equivalently, that $\sum_{i=1}^n x_{e_i}^{-1}[e_i,\gamma] = 0$ for all $\gamma \in H_1(G,\ZZ)$.  Indeed, we may write $\gamma = \sum_{i=1}^n c_i e_i + \gamma'$, where $\gamma'$ is supported away from the edges $e_1,\ldots,e_n$.  The fact that $\partial\gamma = 0$ forces $\sum_{i=1}^n c_i = 0$, so
\begin{equation*}
  \sum_{i=1}^n x_{e_i}^{-1}[e_i,\gamma] = \sum_{i=1}^n c_i = 0.\qedhere
\end{equation*}
\end{proof}

\begin{lemma}
  $\Cer_v(G)$ is trivial in homology.
\end{lemma}
\begin{proof}
  By the construction of $\Cer_v(G)$ and \cref{lem:1}, it suffices to show that $[-1]_{*} \maps H_{1,1}(\Jac(G)) \to H_{1,1}(\Jac(G))$ acts as the identity.  We need only check this on the generators described in \cref{homology-generators}:
\begin{align*}
  [-1]_{*}(\epsilon_i \otimes (0,\lambda_j))
  &= -\epsilon_i \otimes (0,-\lambda_j)\\
  &= -\epsilon_i \otimes (\lambda_j,0) \\
  &= \epsilon_i \otimes (0,\lambda_j) - \partial \left( \epsilon_i \otimes (0,\lambda_j,0) + \epsilon_i \otimes (0,0,0) \right),
\end{align*}
where the second equality follows by equivalence of simplices under translation by $\Lambda$.
\end{proof}

\subsection{Universal Ceresa period}\label{sec:ceresa-zhark-invar}

We take exterior powers of $N_R \cong R^g$ with respect to its $R$-module structure.  We no longer have a topology, let alone a smooth structure for integration, but we can still mimic the effect of $\Theta$ on affine simplices given by \cref{eq:4}.  Therefore, we define $\Theta \maps C_{1,2}(\Jac(G)) \to N_R^{\wedge 3}$ by
\begin{equation*}
\Theta(n \otimes (u_0,u_1,u_2)) \defeq n \wedge (u_1 - u_0) \wedge (u_2 - u_0).
\end{equation*}
A direct computation shows that $\Theta \circ \partial = 0$, so $\sP \defeq \Theta(H_{1,2}(\Jac(G)))$ is well-defined.
Let $\Sigma \in C_{1,2}(\Jac(G))$ be such that $\partial\Sigma = \Cer_v(G)$.  Then the \textit{(universal) Ceresa period} of $G$ is the image of $\Theta(\Sigma)$ in $N_R^{\wedge 3}/\sP$.  Just as for $\alpha(C)$ in \cref{sec:integr-invar}, $\alpha(G)$ depends neither on the choice of $\Sigma$ nor on the basepoint $v \in V(G)$, although the latter fact will not be proved until \cref{prop:1}.  Although we will need the tools of \cref{sec:useful-tools} to compute $\alpha(G)$ efficiently, see \cref{sec:examples} for the minimal examples of nontrivial Ceresa periods.

\subsection{Evaluation}\label{sec:evaluation}

Fix a tropical curve $C$ and a model $(G,\ell)$.  We let $[\cdot,\cdot]$ continue to denote the length pairing with values in $R$ and adopt the notation $[\cdot,\cdot]_C$ for the pairing defined by \cref{eq:30}.  Likewise, we adopt the notations $\pi_C$, $\Lambda_C$, $\Theta_C$, and $\sP_C$ for the ``real edge length'' analogues of the ``universal'' $\pi$, $\Lambda$, $\Theta$, and $\sP$, respectively.  The length function $\ell$ induces a ring map $\ev \maps R \to \RR$ via $x_e \mapsto \ell(e)$.  Then we may write $\Jac(G) = N_R/\Lambda$ and $\Jac(C) = N_{\RR}/\Lambda_C$, where $N = H^1(G,\ZZ)$.  The evaluation map $\ev$ induces a map $N_R \to N_{\RR}$ that we also denote by $\ev$; this in turn induces a map $N_R^{\wedge 3} \to N_{\RR}^{\wedge 3}$.

\begin{proposition}\label{prop:evaluation}
  If $C$ is a tropical curve with underlying graph $G$, then $\ev(\sP) = \sP_C$ and $\ev(\alpha(G)) = \alpha(C)$.  In particular, if $\alpha(G) = 0$, then $\alpha(C) = 0$.
\end{proposition}

\begin{proof}
  By definition, $\ev(N) = N$.  Moreover, it follows from the definitions given in \cref{sec:tropical-jacobians,sec:universal-jacobian} that $[\cdot,\cdot]_C = \ev \circ [\cdot,\cdot]$, hence $\pi_C = \ev \circ \pi$ and $\Lambda_C = \ev(\Lambda)$.  As a result, we obtain a chain map $\ev_{*} \maps C_{\bullet}(\Jac(G)) \to S_{\bullet}(\Jac(C))$ sending $(u_0,\ldots,u_q)$ to the affine $q$-chain $(\ev(u_0),\ldots,\ev(u_q))$, which extends to $(p,q)$-chains by preserving the $N^{\wedge p}$-component.  Comparing \cref{eq:5,ceresa-cycle}, one can check that $\Cer_v(C) = \ev_{*}\Cer_v(G)$.  Then given $\Sigma \in C_{1,2}(\Jac(G))$ for which $\partial\Sigma = \Cer_v(G)$, the affine $(1,2)$-chain $\Sigma_C \defeq \ev_{*}\Sigma$ satisfies $\partial \Sigma_C = \Cer_v(C)$.  It is clear to see that $\ev_{*}$ sends the generators of $H_{1,2}(\Jac(G))$ given by \cref{homology-generators} to the generators of $H_{1,2}(\Jac(C))$ in \cref{eq:36}.  Moreover, \cref{eq:4} ensures that $\Theta_C \circ \ev_{*} = \ev \circ \Theta$; in particular, $\sP_C = \ev(\sP)$ and $\Theta_C(\Sigma_C) = \ev(\Theta(\Sigma))$.  Then $\alpha(C) = \ev(\alpha(G))$.  The second statement follows immediately.
\end{proof}

We naturally associate to a graph $G$ a moduli space of isomorphism classes of tropical curves overlying $G$ as follows (but for more general constructions, see, e.g., \cite[Section~4.1]{chan2013tropical}).  The group $\Aut(G)$ acts on $E(G)$; this induces an action on $\RR_{>0}^{E(G)}$ by permuting the coordinates.  Then let 
\begin{equation*}
  \sM_G^{\mathrm{tr}} \defeq \RR_{>0}^{E(G)}/\Aut(G).
\end{equation*}
Each point of $\sM_G^{\mathrm{tr}}$ determines a length function $\ell$ on $G$ and hence the tropical curve with $(G,\ell)$ as a model.  In analogy to the classical setting, we say that a property is true of a \textit{very general} tropical curve overlying $G$ if it is true of every tropical curve corresponding to a point on $\sM_G^{\mathrm{tr}} \setminus H$, where $H$ is some countable union of hypersurfaces.  Then the second statement of \cref{prop:evaluation} has the following partial converse.
\begin{proposition}\label{prop:evaluation-converse}
  If $\alpha(G) \neq 0$, then a very general tropical curve overlying $G$ has nontrivial Ceresa period.
\end{proposition}
\begin{proof}
  Let $C$ be a tropical curve with $(G,\ell)$ as a model, and adopt the notation used in the proof of \cref{prop:evaluation}.  If we choose generators $p_1,\ldots,p_r$ of $\sP$, then $\ev(p_1),\ldots,\ev(p_r)$ generate $\sP_C$.  Suppose that $\alpha(C) = 0$.  Then by definition, $\Theta_C(\Sigma_C) \in \sP_C$, so we may write 
\begin{equation*}
  \ev(\Theta(\Sigma)) = \Theta_C(\Sigma_C) = \sum_{i=1}^r a_i \ev(p_i) = \ev\left(\sum_{i=1}^r a_ip_i\right)
\end{equation*}
for some $a_i \in \ZZ$, hence 
\begin{equation}\label{eq:33}
S \defeq \Theta(\Sigma) - \sum_{i=1}^r a_ip_i \in \ker\ev.
\end{equation}
Observe that $S \in N_R^{\wedge 3}$.  Since $N_R^{\wedge 3} \cong R \otimes_{\ZZ} N^{\wedge 3}$ is a free $R$-module of finite rank, we may fix a basis and write $S$ as an $R$-linear combination of the basis elements.  Then $S \in \ker\ev$ precisely when all of its finitely many $R$-coefficients vanish under $\ev$.   The fact that $\alpha(G) \neq 0$ implies that $S \neq 0$, so these coefficients cut out a subvariety of $\RR_{> 0}^{E(G)}$ of positive codimension.  The desired statement follows by considering all possible choices of the $a_i$, of which there are countably many.
\end{proof}

\begin{corollary}\label{cor:3}
  If $\alpha(G) \neq 0$, then a very general tropical curve overlying $G$ has nontrivial Ceresa cycle.
\end{corollary}
\begin{proof}
  The result follows immediately from \cref{prop:zharkov,prop:evaluation-converse}.
\end{proof}

We could have defined $\alpha(C)$ directly as $\ev(\alpha(G))$.  In that case, one must show that $\alpha(C)$ is in fact a tropical invariant, i.e., if $C$ has two different models $(G,\ell)$ and $(G',\ell')$, then $\alpha(C)$ is the same whether computed from $\alpha(G)$ or $\alpha(G')$.  We opted for the current approach because it makes the connection to the tropical Ceresa cycle (via \cref{prop:zharkov}) more transparent.

\section{Computational tools}\label{sec:useful-tools}

\subsection{Explicit $(1,2)$-chain $\Upsilon$ whose boundary is $\Cer_v(G)$}\label{sec:explicit-representative}

Given a graph $G$ with basepoint $v \in V(G)$, we shall define a $(1,2)$-chain $\Upsilon$ for which $\partial \Upsilon = \Cer_v(G)$.  To that end, we fix orientations on the edges of $G$ and a spanning tree $T$ and label the edges of $G \setminus T$ by $e_1,\ldots,e_n$.  Each $e_i$ determines a simple cycle $\gamma_i$ that uses only edges from $E(T) \cup \set{e_i}$ and that has positive $e_i$-coefficient.  Order the edges of $\gamma_i$ cyclically.  Then $H_1(G,\ZZ)$ has $\gamma_1,\ldots,\gamma_g$ as a basis; let $\epsilon_1,\ldots,\epsilon_g$ be the dual basis for $N$, i.e.,
\begin{equation*}
  \epsilon_i(\gamma_j) = 
\begin{cases}
  1 & \text{ if } i = j \\
  0 & \text{ otherwise}
\end{cases}.
\end{equation*}
For $e \in E(G)$, the integer $f_i(e) \defeq x_e^{-1}[e,\gamma_i]$ counts the multiplicity of $e$ in $\gamma_i$; by our choice of homology basis, $f_i(e) \in \set{0,\pm 1}$.  Observe also that 
\begin{equation*}
  \pi(e) = \sum_{i=1}^g [e,\gamma_i]\epsilon_i = x_e\sum_{i=1}^gf_i(e)\epsilon_i.
\end{equation*}
From \cref{ceresa-cycle}, we write
\begin{align}
  \Cer_v(G)
  &= \sum_{e\in E(G)} (\Phi_v(e) - [-1]_{*} \Phi_v(e))\nonumber\\
  &= \sum_{e\in E(G)} \left( x_e^{-1}\pi(e) \otimes (\pi(\delta_e),\pi(\delta_e+e)) + x_e^{-1}\pi(e) \otimes (-\pi(\delta_e),-\pi(\delta_e+e)) \right)\nonumber\\
  &= \sum_{e \in E(G)} \sum_{i=1}^gf_i(e)\epsilon_i \otimes \big( (\pi(\delta_e),\pi(\delta_e+e)) + (-\pi(\delta_e),-\pi(\delta_e+e)) \big)\nonumber\\
  &= \sum_{i=1}^g \epsilon_i \otimes \sum_{e \in E(G)} f_i(e) \big((\pi(\delta_e),\pi(\delta_e+e)) + (-\pi(\delta_e),-\pi(\delta_e+e)) \big).\label{eq:1}
\end{align}

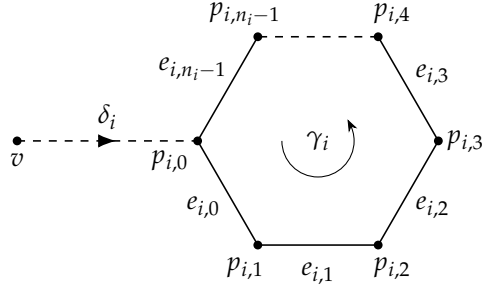
\begin{figure}[htbp]
\centering
\begin{tikzpicture}[semithick, scale=1.6, >=Stealth, decoration={markings,mark=at position 0.5 with {\arrow[xshift={2pt + 2.25\pgflinewidth}]{Latex[scale=1.2]},}}]
  \filldraw (-1.5,0) circle [radius=0.03];
  \draw[postaction={decorate}, dashed] (-1.5,0) node[below=1] {$v$} -- node[above=1] {$\delta_i$} (0,0) node[below left] {$p_{i,0}$};
  \draw (1,0) node {$\gamma_i$};
  \foreach \x in {0,...,5}{
    \filldraw (1,0) +(180+60*\x:1) circle [radius=0.03];
  }
  \foreach \x/\y in {0/1,1/2,2/3,3/4}{
    \draw (1,0) +(180+60*\x:1) -- +(240+60*\x:1);
    \draw (1,0) +(210+60*\x:1.1) node {$e_{i,\x}$};
    \draw (1,0) +(240+60*\x:1.22) node {$p_{i,\y}$};
  }
  \draw[dashed] (1,0) +(60:1) -- +(120:1);
  \draw (1,0) +(120:1) -- +(180:1);
  \draw (1,0) +(150:1.2) node {$e_{i,n_i-1}$};
  \draw (1,0) +(120:1.22) node {$p_{i,n_i-1}$};
  \draw[thin] (1,0) +(180:0.3) arc [start angle=-180, end angle=30, radius=0.3];
  \draw[thin,->] (1,0) ++(30:0.3) -- +(-0.015,0.03);
\end{tikzpicture}

\caption{Notation for the cycle $\gamma_i$\label{fig:cycle}}
\end{figure}

Fix the unique path $\delta_i$ from $v$ to a vertex $p_{i,0}$ in $\gamma_i$ so that $\delta_i$ is contained in $T$ and does not share any edges with $\gamma_i$.  Following the cyclic orientation of $\gamma_i$, label the subsequent vertices by $p_{i,1},p_{i,2},\ldots$ with $p_{i,n_i} = p_{i,0}$.  Write $e_{i,j}$ for the edge in $\gamma_i$ with vertices $p_{i,j}$ and $p_{i,j+1}$ (see \cref{fig:cycle}).  For $0 \leq j \leq n_i$, we define
\begin{equation*}
  u_{i,j} \defeq \pi\left(\delta_i + \sum_{k=0}^{j-1} f_i(e_{i,k})e_{i,k} \right).
\end{equation*}
Notice that $\delta_i + \sum_{k=0}^{j-1} f_i(e_{i,k})e_{i,k}$ is a path from $v$ to $p_{i,j}$ that passes through $p_{i,0},p_{i,1},\ldots,p_{i,j-1}$; in particular, $u_{i,j}$ is a lift of $\Phi_v(p_{i,j})$ to $N_R$, and \cref{eq:1} reduces to
\begin{equation}\label{eq:3}
  \Cer_v(G)
  = \sum_{i=1}^g \epsilon_i \otimes \sum_{j=0}^{n_i-1}  
\begin{cases}
  (u_{i,j},u_{i,j+1}) + (-u_{i,j},-u_{i,j+1}) & \text{if } f_i(e_{i,j}) = 1 \\
  -(u_{i,j+1},u_{i,j}) - (-u_{i,j+1},-u_{i,j}) & \text{if } f_i(e_{i,j}) = -1
\end{cases},
\end{equation}
the $\epsilon_i$-component of which is depicted in \cref{fig:lift}.  This cumbersome notation is necessary because, unlike for simplicial homology, we are not able to permute the vertices of a simplex at the cost of introducing a sign.  The alternative that works in the current context is the following:
\begin{equation}\label{eq:6}  
\begin{split}
  -(u_{i,j+1},u_{i,j}) - (-u_{i,j+1},-u_{i,j}) &= (u_{i,j},u_{i,j+1}) + (-u_{i,j},-u_{i,j+1})- \partial(u_{i,j},u_{i,j+1},u_{i,j}) - \partial(u_{i,j},u_{i,j},u_{i,j})\\
                                               &\quad- \partial(-u_{i,j},-u_{i,j+1},-u_{i,j}) - \partial(-u_{i,j},-u_{i,j},-u_{i,j}).
\end{split}
\end{equation}
In other words, we must add the boundary of a degenerate $(1,2)$-chain in order to permute the vertices of each $(1,1)$-simplex.

Now that we have expressed $\Cer_v(G)$ more explicitly, we are able to define $\Upsilon \defeq \sum_{i=1}^g \epsilon_i \otimes \Upsilon_i$, where
\begin{equation}\label{eq:26}  
\begin{split}
  \Upsilon_i &\defeq \sum_{j=1}^{n_i-1} \left((u_{i,0},u_{i,j},u_{i,j+1}) + (-u_{i,0},-u_{i,j},-u_{i,j+1})\right) + (u_{i,0},u_{i,n_i},-u_{i,0}) + (u_{i,0},-u_{i,0},-u_{i,n_i})\\
             &\quad- \sum_{j \given f_i(e_{i,j}) = -1} \left((u_{i,j},u_{i,j+1},u_{i,j}) + (u_{i,j},u_{i,j},u_{i,j})
               + (-u_{i,j},-u_{i,j+1},-u_{i,j}) + (-u_{i,j},-u_{i,j},-u_{i,j})\right).
\end{split}
\end{equation}
Using the fact that $u_{i,n_i} = u_{i,0} + \lambda_i$, a straightforward computation yields
\begin{equation*}  
\begin{split}
\partial \Upsilon_i &= \sum_{j=0}^{n_i-1} \left((u_{i,j},u_{i,j+1}) + (-u_{i,j},-u_{i,j+1})\right)\\
  &\quad- \sum_{j \given f_i(e_{i,j}) = -1} \partial\left((u_{i,j},u_{i,j+1},u_{i,j}) + (u_{i,j},u_{i,j},u_{i,j})
               + (-u_{i,j},-u_{i,j+1},-u_{i,j}) + (-u_{i,j},-u_{i,j},-u_{i,j})\right),
\end{split}
\end{equation*}
which by \cref{eq:6} equals the $\epsilon_i$-component of \cref{eq:3}.  This is depicted without orientations or degenerate simplices in \cref{fig:lift+upsilon}.  We conclude that $\partial\Upsilon = \Cer_v(G)$.

\begin{remark}\label{rem:1}
  Notice that each $u_{i,j}$ is independent of the orientations on the edges of $G$.  Consequently, $\Upsilon$ is also independent of the orientations.
\end{remark}

\begin{figure}[tbp]
\centering
\subfigure[Support of the $\epsilon_i$-component of $\Cer_v(G)$ in $N_R$\label{fig:lift}]{\begin{tikzpicture}[semithick, scale=2, >=Stealth, decoration={markings,mark=at position 0.5 with {\arrow[xshift={2pt + 2.25\pgflinewidth}]{Latex[scale=1.2]},}}]
  \foreach \x in {0,...,6}{
    \filldraw (0,0) +(180+30*\x:1) circle [radius=0.024];
  }
  \foreach \x in {0,...,4}{
    \draw (0,0) +(180+30*\x:1.2) node {$u_{i,\x}$};
  }
  \foreach \x in {0,...,3}{
    \draw (0,0) +(180+30*\x:1) -- (210+30*\x:1);
  }
  \draw[dashed] (0,0) +(-60:1) -- +(-30:1);
  \draw (0,0) +(-30:1) -- +(0:1);
  \draw (0,0) +(-30:1.33) node {$u_{i,n_i-1}$};
  \draw (0,0) +(0:1.25) node {$u_{i,n_i}$};
  \foreach \x in {0,...,6}{
    \filldraw (0,0.5) +(30*\x:1) circle [radius=0.024];
  }
  \foreach \x in {1,...,4}{
    \draw (0,0.5) +(30*\x:1.18) node {$-u_{i,\x}$};
  }
  \foreach \x in {0,...,3}{
    \draw (0,0.5) +(30*\x:1) -- +(30+30*\x:1);
  }
  \draw[dashed] (0,0.5) +(120:1) -- +(150:1);
  \draw (0,0.5) +(150:1) -- +(180:1);
  \draw (0,0.5) +(150:1.34) node {$-u_{i,n_i-1}$};
  \draw (0,0.5) +(180:1.29) node {$-u_{i,n_i}$};
  \draw (0,0.5) +(0:1.27) node {$-u_{i,0}$};
\end{tikzpicture}
}
\subfigure[Support of $\Upsilon_i$ \label{fig:upsilon}]{\begin{tikzpicture}[semithick, scale=2, >=Stealth, decoration={markings,mark=at position 0.5 with {\arrow[xshift={2pt + 2.25\pgflinewidth}]{Latex[scale=1.2]},}}]
  \foreach \x in {0,...,6}{
    \filldraw (0,0) +(180+30*\x:1) circle [radius=0.024];
    \draw (0,0) +(180:1) -- +(180+30*\x:1);
  }
  \foreach \x in {0,...,4}{
    \draw (0,0) +(180+30*\x:1.2) node {$u_{i,\x}$};
  }
  \foreach \x in {0,...,3}{
    \draw (0,0) +(180+30*\x:1) -- (210+30*\x:1);
  }
  \draw[dashed] (0,0) +(-60:1) -- +(-30:1);
  \draw (0,0) +(-30:1) -- +(0:1);
  \draw (0,0) +(-30:1.33) node {$u_{i,n_i-1}$};
  \draw (0,0) +(0:1.25) node {$u_{i,n_i}$};
  \foreach \x in {0,...,6}{
    \filldraw (0,0.5) +(30*\x:1) circle [radius=0.024];
    \draw (0,0.5) +(0:1) -- +(30*\x:1);
  }
  \foreach \x in {1,...,4}{
    \draw (0,0.5) +(30*\x:1.18) node {$-u_{i,\x}$};
  }
  \foreach \x in {0,...,3}{
    \draw (0,0.5) +(30*\x:1) -- +(30+30*\x:1);
  }
  \draw[dashed] (0,0.5) +(120:1) -- +(150:1);
  \draw (0,0.5) +(150:1) -- +(180:1);
  \draw (0,0.5) +(150:1.34) node {$-u_{i,n_i-1}$};
  \draw (0,0.5) +(180:1.29) node {$-u_{i,n_i}$};
  \draw (0,0.5) +(0:1.27) node {$-u_{i,0}$};
  \draw (0,0) ++(0:1) -- +(0,0.5);
  \draw (0,0) ++(180:1) -- +(0,0.5);
  \draw (0,0) +(180:1) -- +(1,0.5);
\end{tikzpicture}

}
\caption{\label{fig:lift+upsilon}}
\end{figure}

\subsection{Explicit representative $\Theta(\Upsilon)$ for $\alpha(G)$}\label{sec:explicit-computation}

Since $\Theta$ kills degenerate $(1,2)$-simplices, we may ignore the second summation in \cref{eq:26} in the following computation:
\begin{align}
\begin{split}
  \Theta(\Upsilon)
  &= \sum_{i=1}^g \Theta(\epsilon_i \otimes \Upsilon_i)\\
  &= \sum_{i=1}^g\epsilon_i \wedge \left( \sum_{j=1}^{n_i-1} \left( (u_{i,j}-u_{i,0}) \wedge (u_{i,j+1}-u_{i,0}) + (-u_{i,j}+u_{i,0}) \wedge (-u_{i,j+1}+u_{i,0})\right)\right. \\
  &\left.\hphantom{\sum_{i=1}^g\epsilon_i \wedge \Bigg(}\quad+(u_{i,n_i}-u_{i,0}) \wedge (-2u_{i,0}) + (-2u_{i,0}) \wedge (-u_{i,n_i}-u_{i,0})
    \vphantom{\sum_{j=1}^{n_i-1} \left( (u_{i,j}-u_{i,0})\right)}\right)
\end{split}\nonumber\\
  &= \sum_{i=1}^g\epsilon_i \wedge \left( 2\sum_{j=1}^{n_i-1} (u_{i,j}-u_{i,0}) \wedge (u_{i,j+1}-u_{i,j}) - 4 (u_{i,n_i}-u_{i,0}) \wedge u_{i,0} \right)\nonumber\\
  &= \sum_{i=1}^g\epsilon_i \wedge \left( 2\sum_{j=1}^{n_i-1} \left(\sum_{k=0}^{j-1}\pi(f_i(e_{i,k})e_{i,k})\right) \wedge \pi(f_i(e_{i,j})e_{i,j}) - 4 \pi(\gamma_i) \wedge \pi(\delta_i) \right)\nonumber\\
  &= \sum_{i=1}^g\epsilon_i \wedge \left( 2\sum_{0 \leq k < j \leq n_i-1} \pi(f_i(e_{i,k})e_{i,k}) \wedge \pi(f_i(e_{i,j})e_{i,j}) - 4 \pi(\gamma_i) \wedge \pi(\delta_i) \right).\label{eq:13}
\end{align}

Since we have bases $\epsilon_1,\ldots,\epsilon_g$ for $N$ and $\pi(\gamma_1),\ldots,\pi(\gamma_g)$ for $\Lambda$, \cref{homology-generators} tells us that
\begin{equation*}
  H_{1,2}(\Jac(G)) = \ZZ\genset*{\epsilon_i \otimes ((0,\pi(\gamma_j),\pi(\gamma_j+\gamma_k)) - (0,\pi(\gamma_k),\pi(\gamma_j+\gamma_k))) \given i,j,k \in [g],\, j < k}.
\end{equation*}
Applying $\Theta$, we find that
\begin{equation}\label{eq:24}
  \sP = \ZZ\genset*{2\epsilon_i \wedge \pi(\gamma_j) \wedge \pi(\gamma_k) \given i,j,k \in [g],\, j < k}.
\end{equation}
Since $\pi(\gamma_j) = \sum_{m=1}^g [\gamma_j,\gamma_m] \epsilon_m$, we may rewrite
\begin{align}
  2\epsilon_i \wedge \pi(\gamma_j) \wedge \pi(\gamma_k)
  &= 2\sum_{m=1}^g \sum_{n=1}^g [\gamma_j,\gamma_m][\gamma_k,\gamma_n] \epsilon_i \wedge \epsilon_m \wedge \epsilon_n \nonumber\\
  &= 2\sum_{m<n} \left([\gamma_j,\gamma_m][\gamma_k,\gamma_n]-[\gamma_j,\gamma_n][\gamma_k,\gamma_m]\right) \epsilon_i \wedge \epsilon_m \wedge \epsilon_n.\label{eq:25}
\end{align}
In other words, the $\epsilon_i \wedge \epsilon_m \wedge \epsilon_n$-coefficient of $2\epsilon_i \wedge \pi(\gamma_j) \wedge \pi(\gamma_k)$ is twice the $(j,k)\times(m,n)$-minor of the Gram matrix of the pairing $[\cdot,\cdot]$.

\subsection{Basepoint independence of the Ceresa period}\label{sec:basep-indep}

\begin{proposition}\label{lem:basep-indep}
  Fix a graph $G$ and let $\Upsilon$ be the $(1,2)$-chain defined by \cref{eq:26}.  Then $\Theta(\Upsilon)$ does not depend on the choice of basepoint.
\end{proposition}
\begin{proof}
  Since $G$ is connected, it suffices to show that $\Theta(\Upsilon)$ remains unchanged whether computed using $v$ as the basepoint or some vertex $v'$ separated from $v$ by an edge $a$.  Without loss of generality, assume that $a$ is oriented from $v'$ to $v$.

  We observe that, in the computations in \cref{sec:explicit-representative}, as well as in deriving \cref{eq:13}, we do not use any properties of $\delta_i$ other than that it is a path from $v$ to \textit{some} vertex $p_{i,0}$ in the cycle $\gamma_i$.  In particular, if we construct $\Upsilon'$ as in \cref{sec:explicit-representative} using the basepoint $v'$ and the paths $\delta_i' \defeq \delta_i + a$, then the same argument shows that $\partial\Upsilon' = \Cer_{v'}(G)$.  Moreover, the points $p_{i,0}$ have not changed, nor have the edge labels $e_{i,j}$, so it follows from \cref{eq:13} that
\begin{align*}
  \Theta(\Upsilon')
  &= \sum_{i=1}^g\epsilon_i \wedge \left( 2\sum_{0 \leq k < j \leq n_i-1} \pi(f_i(e_{i,k})e_{i,k}) \wedge \pi(f_i(e_{i,j})e_{i,j}) - 4 \pi(\gamma_i) \wedge \pi(\delta_i + a) \right)\\
  &= \Theta(\Upsilon) - 4\sum_{i=1}^g \epsilon_i \wedge \pi(\gamma_i) \wedge \pi(a)\\
  &= \Theta(\Upsilon) - 4\sum_{i=1}^g \epsilon_i \wedge \sum_{j=1}^g [\gamma_i,\gamma_j]\epsilon_j \wedge \pi(a)\\
  &= \Theta(\Upsilon) - 4\sum_{i,j} [\gamma_i,\gamma_j] \epsilon_i \wedge \epsilon_j \wedge \pi(a).
    \intertext{Ordering so that $i < j$ and using the fact that $\epsilon_j \wedge \epsilon_i = -\epsilon_i \wedge \epsilon_j$, we obtain}
  &= \Theta(\Upsilon) - 4\sum_{i < j} \left([\gamma_i,\gamma_j] - [\gamma_j,\gamma_i]\right) \epsilon_i \wedge \epsilon_j \wedge \pi(a).
\end{align*}
Finally, $[\cdot,\cdot]$ is symmetric, so the summation vanishes and we are left with $\Theta(\Upsilon') = \Theta(\Upsilon)$, as desired.
\end{proof}
\begin{corollary}\label{prop:1}
  The Ceresa period $\alpha(G)$ does not depend on the choice of basepoint.
\end{corollary}
\begin{proof}
This follows immediately from \cref{lem:basep-indep} and the fact that $\Theta(\Upsilon)$ is a representative of $\alpha(G)$.
\end{proof}

\begin{corollary}\label{cor:4}
  Let $C$ be a tropical curve overlying $G$.  Then $\alpha(C)$ does not depend on the choice of basepoint.
\end{corollary}
\begin{proof}
  Fix $v, v' \in C$.  After refining $G$ if necessary, we may assume that $v, v' \in V(G)$.  Then \cref{prop:1} implies that $\alpha(G)$ is the same whether computed using $v$ or $v'$.  Meanwhile, $\alpha(C) = \ev(\alpha(G))$ by \cref{prop:evaluation}, and in the proof of that result, it was shown specifically that $\Cer_v(C) = \ev_{*}\Cer_v(G)$; it follows that $\alpha(C)$ is the same whether we obtain it from $v$ or $v'$.
\end{proof}

\subsection{Rewriting $\Theta(\Upsilon)$ in terms of edge pairs}

\begin{lemma}\label{sec:explicit-formula-1}
  The variable $x_{e_i}$ does not appear in any coefficient of $\Theta(\Upsilon)$ for any $i \in [g]$.
\end{lemma}
\begin{proof}
  For each edge $e$, $\pi(e) = x_e\sum_{i=1}^g f_i(e)\epsilon_i \in x_eN$.  In particular, the variable $x_{e_i}$ appears in $\pi(e)$ only for $e = e_i$.  Since no $\delta_j$ contains $e_i$, and $\gamma_j$ contains $e_i$ only for $j = i$, the only computation $\Theta(\epsilon_j \otimes \Upsilon_j)$ in which $\pi(e_i)$ appears is for $j = i$ (see \cref{eq:13}).  But every term of $\Theta(\epsilon_i \otimes \Upsilon_i)$ with $\pi(e_i) = x_{e_i}\epsilon_i$ as a factor also has $\epsilon_i$ as a factor, so the alternating property implies that these terms all vanish.
\end{proof}
It is not hard to see from \cref{eq:13} that $\Theta(\Upsilon)$ is homogeneous of degree 2 in the variables $x_e$.  We now develop a characterization for when a pair of edges $e$ and $e'$ contributes a monomial $x_ex_{e'}\, \epsilon_i \wedge \epsilon_j \wedge \epsilon_k$ to $\Theta(\Upsilon)$.  Recall from \cref{sec:explicit-representative} that
\begin{equation*}
  f_i(e) \defeq x_e^{-1}[e,\gamma_i] = 
\begin{cases}
  0 & \text{ if $e \not\in \supp \gamma_i$}\\
  1 & \text{ if $e$ and $\gamma_i$ have the same orientation}\\
  -1 & \text{ if $e$ and $\gamma_i$ have opposite orientations.}
\end{cases}
\end{equation*}
Likewise, we define
\begin{equation*}
  g_i(e) \defeq x_e^{-1}[e,\delta_i] = 
\begin{cases}
  0 & \text{ if $e \not\in \supp \delta_i$}\\
  1 & \text{ if $e$ and $\delta_i$ have the same orientation}\\
  -1 & \text{ if $e$ and $\delta_i$ have opposite orientations.}
\end{cases}
\end{equation*}
Given distinct edges $e$ and $e'$ in $\gamma_i$, say that $e <_i e'$ if $e$ occurs before $e'$ in the ordering determined by the endpoint of $\delta_i$ and the cyclic orientation of $\gamma_i$.  Equivalently, in the notation of \cref{fig:cycle}, we declare $e_{i,j} <_i e_{i,j+1}$ for $0 \leq j \leq n_i - 2$.  Then let
\begin{equation*}
  h_i(e,e') \defeq 
\begin{cases}
  0 & \text{ if $e \not\in \supp \delta_i$, $e' \not\in \supp \delta_i$, or $e = e'$}\\
  1 & \text{ if $e <_i e'$}\\
  -1 & \text{ if $e' <_i e$.}
\end{cases}
\end{equation*}
Fix a subset $S \subset E(G) \times E(G)$ that contains exactly one element from $\set{(e,e'),(e',e)}$ for each pair of distinct edges $e$ and $e'$.  Likewise, in lieu of requiring that $i < j < k$, we instead fix a subset $S' \subset [g]^3$ that contains exactly one permutation of each tuple $(i,j,k) \in [g]^3$ of all distinct indices.

\begin{proposition}\label{prop:alpha-edge-char}
Fix a graph $G$ and let $\Upsilon$ be the $(1,2)$-chain defined by \cref{eq:26}.  Then we may write
\begin{equation*}
  \Theta(\Upsilon) = \sum_{\substack{(e,e') \in S \\ (i,j,k) \in S'}} a_{i,j,k}(e,e') x_ex_{e'}\,\epsilon_i \wedge \epsilon_j \wedge \epsilon_k,
\end{equation*}
where $a_{i,j,k}(e,e')$ takes values in $\set{0,\pm 2}$ and is nonzero precisely when, up to relabeling $(i,j,k)$, $\gamma_i$ contains both $e$ and $e'$, $\gamma_j$ contains $e$ but not $e'$, and $\gamma_k$ contains $e'$ but not $e$.  In particular,
\begin{equation}\label{eq:29}
  a_{i,j,k}(e,e') = 2f_j(e)f_k(e') \left( f_i(e)f_i(e')h_i(e,e') + 2f_i(e)g_j(e') - 2f_i(e')g_k(e) \right),
\end{equation}
and reversing the cyclic orientation of any one of the cycles $\gamma_i$, $\gamma_j$, or $\gamma_k$ flips the sign.
\end{proposition}
\begin{proof}
  From \cref{eq:13}, we may write 
\begin{align}
  \Theta(\Upsilon)
  &= \sum_{i=1}^g\epsilon_i \wedge \left( 2\sum_{0 \leq k < j \leq n_i-1} \pi(f_i(e_{i,k})e_{i,k}) \wedge \pi(f_i(e_{i,j})e_{i,j}) - 4 \pi(\gamma_i) \wedge \pi(\delta_i) \right)\nonumber\\
  &= 2\sum_{i=1}^g \epsilon_i \wedge \left( \sum_{\substack{e,e' \in \supp\gamma_i \\ e <_i e'}} f_i(e)\pi(e) \wedge f_i(e')\pi(e') - 2\pi(\gamma_i) \wedge \pi(\delta_i) \right)\nonumber\\
  &= 2\sum_{i=1}^g \epsilon_i \wedge \left( \sum_{(e,e') \in S} f_i(e)f_i(e')h_i(e,e') \pi(e) \wedge \pi(e') - 2\pi(\gamma_i) \wedge \pi(\delta_i) \right)\nonumber\\
  \intertext{We may write $\gamma_i = \sum_{e \in E(G)} f_i(e)e$ and $\delta_i = \sum_{e \in E(G)} g_i(e)e$, so that}
  &= 2\sum_{i=1}^g \epsilon_i \wedge \left( \sum_{(e,e') \in S} f_i(e)f_i(e')h_i(e,e') \pi(e) \wedge \pi(e') - 2\sum_{e,e' \in E(G)} f_i(e)g_i(e') \pi(e) \wedge \pi(e') \right)\nonumber\\
  &= \sum_{(e,e') \in S} \sum_{i=1}^g 2\left( f_i(e)f_i(e')h_i(e,e') - 2f_i(e)g_i(e') + 2f_i(e')g_i(e) \right) \epsilon_i \wedge \pi(e) \wedge \pi(e').\nonumber\\
  \intertext{Expanding $\pi(e) = x_e\sum_{j=1}^gf_j(e)\epsilon_j$ and $\pi(e') = x_{e'}\sum_{k=1}^gf_k(e')\epsilon_k$, we obtain}
  &= \sum_{(e,e') \in S} \sum_{i,j,k} 2\left( f_i(e)f_i(e')h_i(e,e') - 2f_i(e)g_i(e') + 2f_i(e')g_i(e) \right) f_j(e) f_k(e') x_ex_{e'}\, \epsilon_i \wedge \epsilon_j \wedge \epsilon_k.\nonumber\\
  \intertext{We reindex by $S'$ to get}
  &= \sum_{(e,e') \in S} \sum_{(i,j,k) \in S'} a_{i,j,k}(e,e') x_ex_{e'}\, \epsilon_i \wedge \epsilon_j \wedge \epsilon_k,\label{eq:2}
\end{align}
where 
\begin{gather*}
  a_{i,j,k}(e,e') \defeq 2(A_iB_{j,k} - A_jB_{i,k} + A_kB_{i,j})\\[7pt]
  A_t \defeq f_t(e)f_t(e')h_t(e,e') - 2f_t(e)g_t(e') + 2f_t(e')g_t(e)\\[7pt]
  B_{m,n} \defeq f_m(e) f_n(e') - f_m(e')f_n(e).
\end{gather*}

As one might expect, reordering the tuple $(e,e')$ does not affect $a_{i,j,k}(e,e')$, while permuting $(i,j,k)$ changes it by the sign of the permutation.  We record here the possible combinations of values of $\abs{f_{\cdot}(\cdot)}$ up to such reordering:
\begin{gather*}
  \begin{array}{ c | c c c }
    \text{(a)} & i & j & k \\
    \hline
    e  & 1 & 1 & * \\
    e' & 1 & 1 & *
  \end{array}
  \qquad
  \begin{array}{ c | c c c }
    \text{(b)} & i & j & k \\
    \hline
    e  & 1 & 1 & 0 \\
    e' & 1 & 0 & 1
  \end{array}
  \qquad
  \begin{array}{ c | c c c }
    \text{(c)} & i & j & k \\
    \hline
    e  & * & 1 & 1 \\
    e' & * & 0 & 0
  \end{array}
  \qquad
  \begin{array}{ c | c c c }
    \text{(d)} & i & j & k \\
    \hline
    e  & * & * & 0 \\
    e' & * & * & 0
  \end{array}.
\end{gather*}
A $*$ indicates that the corresponding entry can have value either $0$ or $1$.  While these cases are not all pairwise disjoint, they do cover all the possibilities.  Indeed, up to permuting the rows and columns, if zero or one of the six values is $0$, then we are in case (a).  If exactly two values are $0$, then we fall into one of (b), (c), or (d).  If three are $0$, then we are in cases (c) or (d).  Finally, if at least four values are $0$, then case (d) applies.

We claim that the $a_{i,j,k}(e,e') \neq 0$ only in case (b).  As a shortcut, we may consider only edges $e$ and $e'$ in $T$; by \cref{sec:explicit-formula-1}, the remaining edges do not contribute terms.  Then $T \setminus \set{e,e'}$ has three connected components; let $G'$ be the graph obtained from $G$ by contracting each of these components to a vertex and removing the additional edges $e_t$ for $t \not\in \set{i,j,k}$.  We depict in \cref{fig:1} the resulting graph for each of the cases (a), (b), and (c) with the edges $e_i,e_j,e_k$ drawn as necessary.  The basepoint $v$ descends to one of the three vertices of $G'$, but by \cref{lem:basep-indep}, we may fix it arbitrarily.  We remark that the values of $f_t$, $g_t$, and $h_t$ on $e$ and $e'$ remain unchanged for $t \in \set{i,j,k}$ when passing from $G$ to $G'$.

The following observations will be helpful in further narrowing down the cases.  By \cref{rem:1}, $\Theta(\Upsilon)$ is independent of the orientations on the edges of $G$.  Then without loss of generality, we may orient the edges of $G'$ as shown in \cref{fig:1}.  Meanwhile, replacing $\gamma_i$ with the cycle $\bar\gamma_i \defeq -\gamma_i$ having the opposite cyclic orientation of edges (i.e., the reverse of the partial order $<_i$) changes the sign of $f_i$ and $h_i$.  Hence, only $A_i$, $B_{i,k}$, and $B_{i,j}$ change sign, thereby causing $a_{i,j,k}(e,e')$ also to change sign.  The same applies to the indices $j$ and $k$.  Notably, this means that the choice for each $t$ of $\gamma_t$ versus $\bar\gamma_t$ does not affect whether or not $a_{i,j,k}(e,e') = 0$, so in \cref{fig:1}, we orient $\gamma_t$ (not depicted) in the same direction as $e_t$.  Finally, recall from \cref{sec:explicit-representative} that the path $\delta_t$ is defined uniquely by the fact that it lies in $T$ and has disjoint support from the cycle $\gamma_t$.

\begin{figure}[h!]
\centering
\subfigure[\label{fig:1a}]{
\begin{tikzpicture}[semithick, scale=1.6, >=Stealth, decoration={markings,mark=at position 0.5 with {\arrow[xshift={2pt + 2.25\pgflinewidth}]{Latex[scale=0.8]},}}]
  \filldraw (0,0) circle [radius=0.03] node[anchor=east] {$v$};
  \filldraw (1,0) circle [radius=0.03];
  \filldraw (2,0) circle [radius=0.03] node[anchor=west] {$\phantom{v}$};
  \draw[postaction={decorate}, very thick] (0,0) -- node[anchor=south] {$e$} (1,0);
  \draw[postaction={decorate}, very thick] (1,0) -- node[anchor=south] {$e'$} (2,0);
  \draw[postaction={decorate}, bend right=60] (2,0) to node[midway, anchor=south] {$e_i$} (0,0);
  \draw[postaction={decorate}, bend left=60] (2,0) to node[midway, anchor=north] {$e_j$} (0,0);
\end{tikzpicture}
}
\subfigure[\label{fig:1b}]{
  \begin{tikzpicture}[semithick, scale=1.6, >=Stealth, decoration={markings,mark=at position 0.5 with {\arrow[xshift={2pt + 2.25\pgflinewidth}]{Latex[scale=0.8]},}}]
  \draw[postaction={decorate}, bend left=60, color=white] (2,0) to node[midway, anchor=north] {$e_j$} (0,0);
  \filldraw (0,0) circle [radius=0.03] node[anchor=east] {$v$};
  \filldraw (1,0) circle [radius=0.03];
  \filldraw (2,0) circle [radius=0.03] node[anchor=west] {$\phantom{v}$};
  \draw[postaction={decorate}, very thick] (0,0) -- node[anchor=south] {$e$} (1,0);
  \draw[postaction={decorate}, very thick] (1,0) -- node[anchor=south] {$e'$} (2,0);
  \draw[postaction={decorate}, bend right=60] (2,0) to node[midway, anchor=south] {$e_i$} (0,0);
  \draw[postaction={decorate}, bend left=80] (1,0) to node[midway, anchor=north] {$e_j$} (0,0);
  \draw[postaction={decorate}, bend left=80] (2,0) to node[midway, anchor=north] {$e_k$} (1,0);
\end{tikzpicture}
}
\subfigure[\label{fig:1c}]{
  \begin{tikzpicture}[semithick, scale=1.6, >=Stealth, decoration={markings,mark=at position 0.5 with {\arrow[xshift={2pt + 2.25\pgflinewidth}]{Latex[scale=0.8]},}}]
  \draw[postaction={decorate}, bend left=60, color=white] (2,0) to node[midway, anchor=north] {$e_j$} (0,0);
  \filldraw (0,0) circle [radius=0.03] node[anchor=east] {$v$};
  \filldraw (1,0) circle [radius=0.03];
  \filldraw (2,0) circle [radius=0.03] node[anchor=west] {$\phantom{v}$};
  \draw[postaction={decorate}, very thick] (0,0) -- node[anchor=south] {$e$} (1,0);
  \draw[postaction={decorate}, very thick] (1,0) -- node[anchor=south] {$e'$} (2,0);
  \draw[postaction={decorate}, bend right=80] (1,0) to node[midway, anchor=south] {$e_j$} (0,0);
  \draw[postaction={decorate}, bend left=80] (1,0) to node[midway, anchor=north] {$e_k$} (0,0);
\end{tikzpicture}
}
\caption[]{}\label{fig:1}
\end{figure}

To reiterate, all of the choices that went into drawing \cref{fig:1} were made without loss of generality up to a sign.  Therefore, one may read off the values of $f_t$, $g_t$, and $h_t$ on $e$ and $e'$ for each of the relevant indices of $t$ from $\set{i,j,k}$; one finds that $a_{i,j,k}(e,e')$ vanishes in cases (a) and (c) and equals $-2$ in case (b).  That $a_{i,j,k}(e,e')$ vanishes in case (d) is not hard to see, since $f_k(e) = f_k(e') = 0$ forces $B_{j,k} = B_{i,k} = A_k = 0$.  The particular expression for the coefficient given by \cref{eq:29} in the statement of the result follows from case (b) by plugging into $a_{i,j,k}(e,e')$ only the values of the indicator functions that vanish.
\end{proof}

\section{Forbidden minor characterization of nontriviality of the Ceresa period}\label{sec:forb-minor-char}

With the powerful tools of \cref{sec:useful-tools} in hand, we now endeavor to prove \cref{thm:A} by following the approach of \cite{corey2022ceresa}.  In \cref{sec:examples}, we show that $K_4$ and $L_3$, the graphs depicted in \cref{fig:graphs}, have nontrivial Ceresa period.  In \cref{sec:contraction,sec:deletion}, we show that having nontrivial Ceresa period is preserved under certain contraction and deletion operations, respectively.  Finally, in \cref{sec:graphs-hyper-type}, we show that graphs of hyperelliptic type have trivial Ceresa period and that there are no other cases left to consider.

Throughout this section, whenever we are working with two graphs $G$ and $G'$, for each object $O$ associated to $G$ that was defined in \cref{sec:definition,sec:useful-tools}, we let $O'$ denote the corresponding object for $G'$.

\begin{figure}[htbp]
\centering
\subfigure[$K_4$\label{fig:k4}]{\begin{tikzpicture}[semithick, scale=1.6, >=Stealth, decoration={markings,mark=at position 0.5 with {\arrow[xshift={2pt + 2.25\pgflinewidth}]{Latex[scale=1.2]},}}]
  \filldraw (0,0) circle [radius=0.03];
  \foreach \x in {0,1,2}{
    \filldraw (120*\x+90:1) circle [radius=0.03];
  }
  \draw[postaction={decorate}] (210:1) arc [start angle=210, end angle=90, radius=1];
  \draw (150:1.2) node {$e_1$};
  \draw[postaction={decorate}] (90:1) arc [start angle=90, end angle=-30, radius=1];
  \draw (270:1.2) node {$e_2$};
  \draw[postaction={decorate}] (330:1) arc [start angle=330, end angle=210, radius=1];
  \draw (30:1.2) node {$e_3$};
  \draw[postaction={decorate}] (0,0) -- (330:1);
  \draw (355:0.5) node {$e_4$};
  \draw[postaction={decorate}] (0,0) -- (90:1);
  \draw (65:0.5) node {$e_5$};
  \draw[postaction={decorate}] (0,0) -- (210:1);
  \draw (185:0.5) node {$e_6$};
  \draw[very thick] (0,0) -- (330:1);
  \draw[very thick] (0,0) -- (90:1);
  \draw[very thick] (0,0) -- (210:1);
  \draw (0,-0.25) node {$v$};
\end{tikzpicture}

}
\subfigure[$L_3$\label{fig:l3}]{\begin{tikzpicture}[semithick, xscale=-1, scale=1.6, >=Stealth, decoration={markings,mark=at position 0.5 with {\arrow[xshift={2pt + 2.25\pgflinewidth}]{Latex[scale=1.2]},}}]
  \foreach \x in {0,1,2}{
    \filldraw (120*\x+90:1) circle [radius=0.03];
  }
  \draw[postaction={decorate}] (90:1) arc [start angle=90, end angle=-30, radius=1];
  \draw (30:1.2) node {$e_1$};
  \draw[postaction={decorate}] (330:1) arc [start angle=330, end angle=210, radius=1];
  \draw (270:1.2) node {$e_2$};
  \draw[postaction={decorate}] (90:1) arc [start angle=90, end angle=210, radius=1];
  \draw (150:1.2) node {$e_3$};
  \draw[postaction={decorate}] (90:1) -- (210:1);
  \draw (150:0.31) node {$e_4$};
  \draw[postaction={decorate}] (330:1) -- (90:1);
  \draw (30:0.31) node {$e_6$};
  \draw[very thick] (330:1) -- (90:1);
  \draw[postaction={decorate}] (210:1) -- (330:1);
  \draw (270:0.31) node {$e_5$};
  \draw[very thick] (210:1) -- (330:1);
  \draw (-30:1.2) node {$v$};
\end{tikzpicture}

}
\caption{Minimal graphs with nontrivial Ceresa period\label{fig:graphs}}
\end{figure}
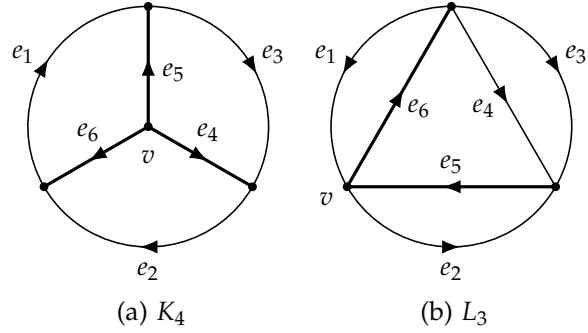

\subsection{Base cases}\label{sec:examples}

\begin{example}\label{ex:k4}
  Let $G \defeq K_4$ be the graph with basepoint $v$ and oriented edges $e_i$ as depicted in \cref{fig:k4}, with $T = \set{e_4,e_5,e_6}$, paths $\delta_i = 0$ for all $i$, and the cyclic orientation of $\gamma_i$ chosen to match the orientation of $e_i$.  Let $x_i \defeq x_{e_i}$.  Then $[\cdot,\cdot]$ has Gram matrix
\begin{equation*}  
\begin{pmatrix}
  x_1+x_5+x_6 & -x_6 & -x_5 \\
  -x_6 & x_2+x_4+x_6 & -x_4 \\
  -x_5 & -x_4 & x_3+x_4+x_5.
\end{pmatrix}.
\end{equation*}
One can show using either \cref{eq:13} or \cref{prop:alpha-edge-char} that 
\begin{equation*}
  \Theta(\Upsilon) = 2(x_4x_5+x_4x_6+x_5x_6)\,\epsilon_1 \wedge \epsilon_2 \wedge \epsilon_3.
\end{equation*}
Meanwhile, \cref{eq:24,eq:25} imply that $\sP$ is generated by the elements
\begin{align*}
  2(x_{2} x_{5} + x_{4} x_{5} + x_{4} x_{6} + x_{5} x_{6})\,&\epsilon_1 \wedge \epsilon_2 \wedge \epsilon_3, \\
  2(-x_{4} x_{5} - x_{3} x_{6} - x_{4} x_{6} - x_{5} x_{6})\,&\epsilon_1 \wedge \epsilon_2 \wedge \epsilon_3, \\
  2(x_{2} x_{3} + x_{2} x_{4} + x_{3} x_{4} + x_{2} x_{5} + x_{4} x_{5} + x_{3} x_{6} + x_{4} x_{6} + x_{5} x_{6})\,&\epsilon_1 \wedge \epsilon_2 \wedge \epsilon_3 \\
  2(x_{1} x_{4} + x_{4} x_{5} + x_{4} x_{6} + x_{5} x_{6})\,&\epsilon_1 \wedge \epsilon_2 \wedge \epsilon_3, \\
  2(-x_{1} x_{3} - x_{1} x_{4} - x_{1} x_{5} - x_{3} x_{5} - x_{4} x_{5} - x_{3} x_{6} - x_{4} x_{6} - x_{5} x_{6})\,&\epsilon_1 \wedge \epsilon_2 \wedge \epsilon_3, \\
  2(x_{4} x_{5} + x_{3} x_{6} + x_{4} x_{6} + x_{5} x_{6})\,&\epsilon_1 \wedge \epsilon_2 \wedge \epsilon_3, \\
  2(x_{1} x_{2} + x_{1} x_{4} + x_{2} x_{5} + x_{4} x_{5} + x_{1} x_{6} + x_{2} x_{6} + x_{4} x_{6} + x_{5} x_{6})\,&\epsilon_1 \wedge \epsilon_2 \wedge \epsilon_3, \\
  2(-x_{1} x_{4} - x_{4} x_{5} - x_{4} x_{6} - x_{5} x_{6})\,&\epsilon_1 \wedge \epsilon_2 \wedge \epsilon_3, \\
  2(x_{2} x_{5} + x_{4} x_{5} + x_{4} x_{6} + x_{5} x_{6})\,&\epsilon_1 \wedge \epsilon_2 \wedge \epsilon_3.
\end{align*}
Simplifying, we are left with six generators:
\begin{align*}
  2(x_{1} x_{4} - x_{3} x_{6})\,&\epsilon_1 \wedge \epsilon_2 \wedge \epsilon_3 \\  
  2(x_{2} x_{5} - x_{3} x_{6})\,&\epsilon_1 \wedge \epsilon_2 \wedge \epsilon_3 \\
  2(x_{1} x_{2} + x_{2} x_{5} + x_{1} x_{6} + x_{2} x_{6})\,&\epsilon_1 \wedge \epsilon_2 \wedge \epsilon_3 \\
  2(x_{1} x_{3} + x_{1} x_{5} + x_{3} x_{5} + x_{3} x_{6})\,&\epsilon_1 \wedge \epsilon_2 \wedge \epsilon_3 \\
  2(x_{2} x_{3} + x_{2} x_{4} + x_{3} x_{4} + x_{3} x_{6})\,&\epsilon_1 \wedge \epsilon_2 \wedge \epsilon_3 \\
  2(x_{4} x_{5} + x_{3} x_{6} + x_{4} x_{6} + x_{5} x_{6})\,&\epsilon_1 \wedge \epsilon_2 \wedge \epsilon_3.
\end{align*}
Suppose that we could obtain $\Theta(\Upsilon)$ as some $\ZZ$-linear combination of these generators.  Since $x_1x_4$, $x_1x_2$, $x_1x_3$, and $x_2x_3$ each appear in only one generator and not in $\Theta(\Upsilon)$, those generators cannot contribute.  This leaves the second and the sixth.  Continuing along the same lines, the second generator uniquely contains the term $x_2x_5$, so it must be trivial.  Now the sixth generator uniquely contains $x_3x_6$, so it also vanishes, a contradiction.  In other words, $\alpha(K_4) \neq 0$, with 
\begin{equation*}
\alpha(K_4) = 2(x_4x_5+x_4x_6+x_5x_6)\,\epsilon_1 \wedge \epsilon_2 \wedge \epsilon_3 \pmod{\sP}.
\end{equation*}
\end{example}

\begin{example}\label{ex:l3}
  Consider the graph $G \defeq L_3$ with basepoint $v$ and oriented edges $e_i$ as shown in \cref{fig:l3}, with $T = \set{e_5,e_6}$, paths $\delta_i = 0$ for all $i$, and the cyclic orientation of $\gamma_i$ matching $e_i$.  We again abbreviate $x_i \defeq x_{e_i}$.  Then $[\cdot,\cdot]$ has Gram matrix
\begin{equation*}
\begin{pmatrix}
  x_1+x_6 & 0 & x_6 & x_6 \\
  0 & x_2+x_5 & x_5 & x_5 \\
  x_6 & x_5   & x_3+x_5+x_6 & x_5+x_6 \\
  x_6 & x_5   & x_5+x_6 & x_4+x_5+x_6
\end{pmatrix}.
\end{equation*}
One can show that
\begin{equation*}
  \Theta(\Upsilon) = 2x_5x_6(\epsilon_1 \wedge \epsilon_2 \wedge \epsilon_4 + \epsilon_1 \wedge \epsilon_2 \wedge \epsilon_3).
\end{equation*}
By \cref{prop:evaluation}, to prove that $\alpha(L_3) \neq 0$, it suffices to show that some tropical curve overlying $L_3$ has nontrivial Ceresa period.  Indeed, let $C$ be the tropical curve obtained by evaluating every edge length $x_i$ to $1$, with the corresponding map $\ev \maps N_R \to N_{\RR}$.  Then $\sP_C$ is generated by the four elements
\begin{alignat*}{4}
  2\epsilon_2 \wedge \epsilon_3 \wedge \epsilon_4 && && {}+ 2\epsilon_1 \wedge \epsilon_2 \wedge \epsilon_4 && {}+ 2\epsilon_1 \wedge \epsilon_2 \wedge \epsilon_3 &, \\
  && \phantom{{}+{}} 2\epsilon_1 \wedge \epsilon_3 \wedge \epsilon_4 && {}+ 2\epsilon_1 \wedge \epsilon_2 \wedge \epsilon_4 && {}+ 2\epsilon_1 \wedge \epsilon_2 \wedge \epsilon_3 &, \\
  && && 4\epsilon_1 \wedge \epsilon_2 \wedge \epsilon_4 && &,\\
  && && && 4\epsilon_1 \wedge \epsilon_2 \wedge \epsilon_3 &.
\end{alignat*}
Let $\Upsilon_C \defeq \ev_{*}\Upsilon$.  Then $\Theta(\Upsilon_C) = 2\epsilon_1 \wedge \epsilon_2 \wedge \epsilon_4 + 2\epsilon_1 \wedge \epsilon_2 \wedge \epsilon_3$, which is clearly not in $\sP_C$, as desired.  Therefore, $\alpha(L_3) \neq 0$, with 
\begin{equation*}
  \alpha(L_3) = 2x_5x_6(\epsilon_1 \wedge \epsilon_2 \wedge \epsilon_4 + \epsilon_1 \wedge \epsilon_2 \wedge \epsilon_3) \pmod {\sP}.
\end{equation*}
\end{example}

\subsection{Contraction}\label{sec:contraction}

Given a graph $G$ with a non-loop edge $a$, we let $G / a$ denote the graph obtained from $G$ by contracting $a$.  We identify $E(G) = E(G / a) \sqcup \set{a}$ in the usual way.

\begin{proposition}\label{contraction}
  Let $G$ be a graph with a non-loop edge $a$.  If $\alpha(G)$, then $\alpha(G/a) = 0$.  Moreover, if $a$ is a separating edge, then the converse also holds.
\end{proposition}

\noindent Although we do not need to use the explicit $(1,2)$-chain $\Upsilon$ to prove the first part of this result, it is nonetheless helpful in making the argument simpler.

\begin{proof}

Let $G' \defeq G / a$.  We shall declare $v \defeq a^-$ to be the basepoint of both $G$ and $G'$.  Fix a spanning tree $T$ of $G$ containing $a$ and let $T' \defeq T / a$ be the corresponding spanning tree in $G'$.  Let $\gamma_1,\ldots,\gamma_g$ be the basis of $H_1(G,\ZZ) \cong H_1(G',\ZZ)$ determined by $T$ in $G$ and $T'$ in $G'$.  Finally, choose orientations of the edges and cycles in $G$; these descend to $G'$, allowing us to define both $\Upsilon$ and $\Upsilon'$ as in \cref{sec:explicit-representative}.

Let $\iota \maps R' \to R$ be the natural inclusion, which misses $x_a$, and define a projection map $\rho \maps R \to R'$ via $x_a \mapsto 0$ and $x_e \mapsto x_e$ for $e \neq a$.  Notice that $\rho \circ \iota = \id$.  Then the induced maps $\iota \maps N'_{R'} \to N_R$ and $\rho \maps N_R \to N'_{R'}$ making the natural identification $N \cong N'$ also satisfy $\rho \circ \iota = \id$.  We write $c \maps C_1(G,\ZZ) \to C_1(G',\ZZ)$ for the homomorphism sending $a \mapsto 0$ and $e \mapsto e$ for $e \neq a$.

It is straightforward to check that $\rho([e,e']) = [c(e),c(e')]'$ for all $e,e' \in E(G)$, hence $\rho \circ \pi = \pi' \circ c$.  In particular, 
\begin{equation*}
\rho(\epsilon_i \wedge \pi(\gamma_j) \wedge \pi(\gamma_k)) = \epsilon_i \wedge \pi'(\gamma_j) \wedge \pi'(\gamma_k);
\end{equation*}
by \cref{eq:24}, $\rho$ induces a group isomorphism $\sP \cong \sP'$.  Moreover, it is clear from the edge pair characterization in \cref{prop:alpha-edge-char} that, for edges $e,e' \in E(G')$, the indicator functions $f_t$, $g_t$, and $h_t$ are the same in $G'$ as they are in $G$.  In other words,
\begin{equation*}
  \Theta(\Upsilon) = \iota(\Theta'(\Upsilon')) + \sum_{\substack{e \in E(G) \\ (i,j,k) \in S'}} a_{i,j,k}(e,a) x_ex_a \epsilon_i \wedge \epsilon_j \wedge \epsilon_k.
\end{equation*}
Then $\rho(\Theta(\Upsilon)) = \Theta'(\Upsilon')$, proving the first part of the statement.  If $a$ is a separating edge, then $a$ is not part of any cycle.  This implies that $x_a$ does not appear in $\pi(\gamma_t)$ for any $t$, so $\iota$ induces $\sP' \cong \sP$.  Likewise, $a_{i,j,k}(e,a) = 0$ for all $e \in E(G)$, so $\iota(\Theta'(\Upsilon')) = \Theta(\Upsilon)$.  This proves the remaining part of the statement.
\end{proof}

\subsection{Deletion}\label{sec:deletion}

Fix a graph $G$ and let $\Upsilon$ be the explicit $(1,2)$-chain defined in \cref{sec:explicit-representative}.  If $\alpha(G) = 0$, then by definition, we can write $\Theta(\Upsilon)$ as some $\ZZ$-linear combination of the generators in \cref{eq:24}.  In fact, not all of the generators are necessary.
\begin{lemma}\label{sec:deletion-1}
  If $\alpha(G) = 0$, then $\Theta(\Upsilon)$ is a $\ZZ$-linear combination of the generators 
\begin{equation*}
  \set{2\epsilon_i \wedge \pi(\gamma_i) \wedge \pi(\gamma_j) \given i, j \in [g],\, i\neq j}.
\end{equation*}
\end{lemma}
\begin{proof}
  Fix indices $i$, $j$, and $k$ all distinct with $j < k$.  By \cref{eq:25}, the generator $2\epsilon_i \wedge \pi(\gamma_j) \wedge \pi(\gamma_k)$ contributes a term 
\begin{equation*}
  2 \left( [\gamma_j,\gamma_j][\gamma_k,\gamma_k] - [\gamma_j,\gamma_k]^2 \right)\epsilon_i \wedge \epsilon_j \wedge \epsilon_k,
\end{equation*}
which contains $2x_{e_j}x_{e_k}\, \epsilon_i \wedge \epsilon_j \wedge \epsilon_k$.  Since $e_t$ appears only in $\gamma_t$, it is straightforward to show that this is in fact the only generator that contains a nonzero multiple of $2x_{e_j}x_{e_k} \epsilon_i \wedge \epsilon_j \wedge \epsilon_k$.  Moreover, by \cref{sec:explicit-formula-1}, such a term also does not appear in $\Theta(\Upsilon)$; as there is no way to cancel this term out with a different generator, we conclude that $\epsilon_i \wedge \pi(\gamma_j) \wedge \pi(\gamma_k)$ cannot contribute to a combination equaling $\Theta(\Upsilon)$.  In other words, we must have either $i = j$ or $i = k$.
\end{proof}  

\begin{proposition}\label{deletion}
  Let $G$ be a graph with an edge $a$.  If either 
  \begin{enumerate}
  \item \label{item:4} $a$ is a loop edge or
  \item \label{item:5} $a$ is parallel to an edge $a'$,
  \end{enumerate}
then $\alpha(G) = 0$ implies that $\alpha(G \setminus a) = 0$.
\end{proposition}

\begin{proof}
  Let $G' \defeq G \setminus a$.  In either case \ref{item:4} or \ref{item:5}, $a$ is part of some cycle.  Without loss of generality, we may choose the spanning tree $T$ and the labeling on the edges of $G \setminus T$ so that $e_g = a$.  We also fix a basepoint $v$ and orientations on the edges and cycles.  These choices descend to $G'$.  Identifying edges of $G'$ with the corresponding edges of $G$, there is a natural inclusion $\phi \maps R' \to R$ that misses $x_{e_g}$.  Likewise, identifying the cycles of $G'$ with the corresponding cycles in $G$, we let $\iota \maps H_1(G',\ZZ) \to H_1(G,\ZZ)$ denote the induced inclusion on homology and $\rho \maps H_1(G,\ZZ) \to H_1(G',\ZZ)$ the projection that kills $\gamma_g$. Define
\begin{equation*}
\begin{tikzcd}[row sep=0pt]
  N'_{R'} \ar[r, "\nu"] & N_R \\
  u \ar[r, mapsto] & \phi \circ u \circ \rho
\end{tikzcd}
\qquad
\begin{tikzcd}[row sep=0pt]
  N_R \ar[r, "\xi"] & N_R \\
  u \ar[r, mapsto] & u \circ \iota \circ \rho
\end{tikzcd}.
\end{equation*}
  In coordinates, we may identify $N'_{R'} \cong R'\genset{\epsilon_1,\ldots,\epsilon_{g-1}}$ and $N_R \cong R\genset{\epsilon_1,\ldots,\epsilon_g}$; then $\nu$ is the inclusion of $R'$-modules sending $\epsilon_i \mapsto \epsilon_i$, while $\xi$ is the projection of $R$-modules sending $\epsilon_i \mapsto \epsilon_i$ for $i < g$ and $\epsilon_g \mapsto 0$.  

  Because $x_{e_g}$ appears in $[e,\gamma_i]$ only for $e = e_g$ and $i = g$, it is straightforward to check that
\begin{equation*}
  \nu(\pi'(e)) = \xi(\pi(e))
\end{equation*}
for all edges $e \in E(G') \subset E(G)$.  We claim that $\nu_{*}\Upsilon' = \xi_{*}\Upsilon$.  Indeed, observe first that $\Upsilon'$ does not have an $\epsilon_g$-component, and that the $\epsilon_g$-component of $\Upsilon$ is killed by $\xi_{*}$.  Fixing $i \neq g$, we recall that the paths $\delta_i$ defined in \cref{sec:explicit-representative} remain in $T = T'$, so $\delta_i' = \delta_i$.  The definition of $\Upsilon'$ given by \cref{eq:26} depends only on the points $u_{i,j}' = \pi'\left(\delta_i + \sum_{k=0}^{j-1} f_i(e_{i,k})e_{i,k} \right) \in N'_{R'}$.  Then $\nu(u_{i,j}') = \xi(u_{i,j})$ for all $j$, so the definition of the pushforward in \cref{eq:23} implies that $\nu_{*}(\epsilon_i \otimes \Upsilon_i') = \xi_{*}(\epsilon_i \otimes \Upsilon_i)$.  The claim follows.  Consequently,
\begin{equation*}
  \nu(\Theta'(\Upsilon')) = \Theta(\nu_{*}\Upsilon') = \Theta(\xi_{*}\Upsilon) = \xi(\Theta(\Upsilon)).
\end{equation*}

A similar computation on generators of $H_{1,2}(\Jac(G))$ shows that $\nu(\sP') \subset \xi(\sP)$, but we may not have equality in general.  Explicitly,
\begin{equation*}
\nu(\sP') = \ZZ\genset*{2\epsilon_i \wedge \xi(\pi(\gamma_j)) \wedge \xi(\pi(\gamma_k)) \given i,j,k \in [g-1],\, j < k},
\end{equation*}
leaving open the possibility that $\xi$ fails to map generators of the form $2\epsilon_i \wedge \pi(\gamma_j) \wedge \pi(\gamma_g)$ into $\nu(\sP')$.  Here we have restricted our attention to the case where $i < g$, since $\xi(\epsilon_g) = 0$.  Our assumption that $\alpha(G) = 0$ means that we may write $\Theta(\Upsilon)$ explicitly as a $\ZZ$-linear combination of the generators in \cref{sec:deletion-1}.  Therefore, if we can show that $2\epsilon_i \wedge \xi(\pi(\gamma_i)) \wedge \xi(\pi(\gamma_g))$ lies in $\nu(\sP')$ for any such generator appearing in the linear combination, then we will have shown that $\nu(\Theta'(\Upsilon')) \in \nu(\sP')$; by injectivity of $\nu$, this would imply in turn that $\Theta'(\Upsilon') \in \sP'$, as desired.

For case \ref{item:4}, this is straightforward because $\pi(\gamma_g) = \pi(e_g) = x_{e_g}\epsilon_g$, so $\xi(\pi(\gamma_g)) = 0$.  In case \ref{item:5}, recall that $e_g$ is parallel to an edge $a'$.  Without loss of generality, assume that $e_g$ and $a'$ have the same orientation.  If $a'$ is a separating edge in $G'$, then $\gamma_g = -a' + e_g$ and $a'$ is not part of any other cycle $\gamma_i$.  In particular, $\pi(\gamma_g) = (x_{a'} + x_{e_g})\epsilon_g$, so we again have that $\xi(\pi(\gamma_g)) = 0$.  Otherwise, $a'$ is not separating in $G'$, so we may choose $T$ so that $e_{g-1} = a'$.  We write explicitly 
\begin{equation*}
  \Theta(\Upsilon) \equiv 2 \sum_{i=1}^{g-1} a_i \epsilon_i \wedge \pi(\gamma_i) \wedge \pi(\gamma_g) \pmod{\xi^{-1}(\nu(\sP'))}
\end{equation*}
for $a_i \in \ZZ$, where we have omitted from the linear combination the generators of $H_{1,2}(\Jac(G))$ that we already know map to $\nu(\sP')$ under $\xi$, i.e., those that do not contain $\pi(\gamma_g)$ as a factor or that have $\epsilon_g$ as the first factor.

The fact that $\gamma_g = \gamma_{g-1} - e_{g-1} + e_g$ allows us to rewrite
\begin{equation}\label{eq:28}
\Theta(\Upsilon) \equiv - 2 x_{e_{g-1}} \sum_{i=1}^{g-1} a_i \epsilon_i \wedge \pi(\gamma_i) \wedge \epsilon_{g-1} + 2 x_{e_g} \sum_{i=1}^{g-1} a_i \epsilon_i \wedge \pi(\gamma_i) \wedge \epsilon_g \pmod{\xi^{-1}(\nu(\sP'))}.
\end{equation}
Since $x_{e_g}$ appears as a coefficient in $\pi(\gamma_j)$ only for $j = g$, the only place where it appears in \cref{eq:28} is where we have explicitly written it before the second summation.  In particular, $x_{e_g}$ does not appear in any of the omitted generators of the form $2\epsilon_g \wedge \pi(\gamma_g) \wedge \pi(\gamma_j)$ because the leading factor of $\epsilon_g$ kills the $\epsilon_g$-term of $\pi(\gamma_g)$, nor does it appear in $\Theta(\Upsilon)$ by \cref{sec:explicit-formula-1}.  We conclude that
\begin{align*}
  0
  &= \sum_{i=1}^{g-1} a_i \epsilon_i \wedge \pi(\gamma_i) \wedge \epsilon_g\\
  &= \sum_{i=1}^{g-1} a_i \epsilon_i \wedge \left(\sum_{j=1}^g[\gamma_i,\gamma_j]\epsilon_j\right) \wedge \epsilon_g\\
  &= \sum_{i=1}^{g-1}\sum_{j=1}^{g-1} [\gamma_i,\gamma_j] a_i \epsilon_i \wedge \epsilon_j \wedge \epsilon_g,\\
  &= \sum_{1 \leq i < j \leq g-1} [\gamma_i,\gamma_j] (a_i-a_j) \epsilon_i \wedge \epsilon_j \wedge \epsilon_g,
\end{align*}
hence $[\gamma_i,\gamma_j](a_i - a_j) = 0$ for all $i,j \in [g-1]$.  Applying $\xi$ to \cref{eq:28} kills terms that contain $\epsilon_g$ as a factor, so we obtain
\begin{align*}
  \xi(\Theta(\Upsilon))
  &\equiv \xi\left(-2 x_{e_{g-1}} \sum_{i=1}^{g-1} a_i \epsilon_i \wedge \pi(\gamma_i) \wedge \epsilon_{g-1}\right) \pmod{\nu(\sP')}\\
  &= \xi\left(-2 x_{e_{g-1}} \sum_{i=1}^{g-1}\sum_{j=1}^g [\gamma_i,\gamma_j] a_i \epsilon_i \wedge \epsilon_j \wedge \epsilon_{g-1}\right)\\
  &= \xi\left(-2 x_{e_{g-1}} \sum_{i=1}^{g-1}\sum_{j=i+1}^g [\gamma_i,\gamma_j] (a_i-a_j) \epsilon_i \wedge \epsilon_j \wedge \epsilon_{g-1}\right)\\
  &= \xi\left(-2 x_{e_{g-1}} \sum_{i=1}^{g-1} [\gamma_i,\gamma_g] (a_i-a_g) \epsilon_i \wedge \epsilon_g \wedge \epsilon_{g-1}\right)\\
  &= 0,
\end{align*}
completing the proof.
\end{proof}

\subsection{Graphs of hyperelliptic type}\label{sec:graphs-hyper-type}

Given a graph $G$, we write $G^2$ to mean the \textit{2-edge-connectivization} of $G$, obtained by contracting each of the separating edges of $G$.
\begin{corollary}\label{lem:2}
  $\alpha(G) = 0$ if and only if $\alpha(G^2) = 0$.
\end{corollary}
\begin{proof}
  This follows immediately from repeated application of \cref{contraction}.
\end{proof}

\begin{lemma}\label{lem:3}
  Let $G_1$ and $G_2$ be graphs, with $G \defeq G_1 \vee G_2$ the wedge sum.  If $G_1$ and $G_2$ are both period-trivial, then so is $G$.  In particular, a graph has trivial Ceresa period if each of its maximal 2-connected components is.
\end{lemma}
\begin{proof}
  The second statement follows trivially from the first.  To prove the first statement, let $v_1 \in V(G_1)$ and $v_2 \in V(G_2)$ be the two vertices identified in $G$.  For $t \in \set{1,2}$, fix $v_t$ as the basepoint of $G_t$ with spanning tree $T_t$.  Notice that $T \defeq T_1 \vee T_2$ is a spanning tree of $G$; let $v_1 = v_2$ be the basepoint for $G$.  Let $\gamma_{t,1},\ldots,\gamma_{t,g_t}$ be the simple cycles determined by $T_t$ with arbitrary orientation; then $\gamma_{1,1},\ldots,\gamma_{1,g_1},\gamma_{2,1},\ldots,\gamma_{2,g_2}$ are the simple cycles of $G$.  Consequently, we identify $H_1(G,\ZZ) \cong H_1(G_1,\ZZ) \oplus H_1(G_2,\ZZ)$.  The paths $\delta_{t,i}$ in $T_t$ descend to the corresponding paths in $T$, allowing us to define $(1,2)$-chains $\Upsilon_1$, $\Upsilon_2$, and $\Upsilon$ using the explicit construction given in \cref{sec:explicit-representative}.

  We may regard $R$ as the coproduct of $\ZZ$-algebras $R \cong R_1 \oplus_{\ZZ} R_2$.  Then $N_R \cong N_{1,R} \oplus N_{2,R}$ as $R$-modules.  It is straightforward to see using \cref{prop:alpha-edge-char} that $\Theta(\Upsilon)$ may be obtained as $\Theta(\Upsilon_1) + \Theta(\Upsilon_2)$.  Indeed, no pair of edges $(e_1,e_2)$ with $e_t \in E(G_t)$ share a common cycle, so they do not contribute to $\Theta(\Upsilon)$.  Meanwhile, any pairs $(e,e')$ with both edges coming from the same subgraph $G_1$ have the same $f_{1,i}$, $g_{1,i}$, and $h_{1,i}$ values in $G$ as they do in $G_1$, with $f_{2,i}$, $g_{2,i}$, and $h_{2,i}$ values all zero (and vice versa for edge pairs coming from $G_2$).  Likewise, one can see from the generators given by \cref{eq:24} that $\sP_1 \oplus \sP_2 \subset \sP$; it follows that if $\Theta(\Upsilon_t) \in \sP_t$ for both $t$, then $\Theta(\Upsilon) \in \sP$, as desired.
\end{proof}

Let $G$ and $G'$ be graphs.  We say that $G'$ is a \textit{permissible minor} of $G$ if we may obtain $G'$ from $G$ by deleting only loops or parallel edges and contracting only non-loop edges.  Then \cref{contraction,deletion} immediately imply that $\alpha(G) = 0$ only if $\alpha(G') = 0$.

A tropical curve $C$ is \textit{hyperelliptic} if it admits an involution $\iota$ for which the quotient $C/\iota$ is a tree.  More generally, it is of \textit{hyperelliptic type} if its Jacobian is isomorphic to that of a hyperelliptic tropical curve.  We say that a graph $G$ is of \textit{hyperelliptic type} if some choice of edge lengths makes it into a hyperelliptic-type tropical curve.  For more details on these notions, we refer to \cite{baker2009harmonic,amini2015lifting-I,amini2015lifting-II,chan2013tropical}.  By \cite[Proposition~3.3]{corey2021tropical}, the property of being of hyperelliptic type does not in fact depend on the choice of edge lengths.  Following \cite{corey2021tropical}, we say that $G$ is \textit{strongly of hyperelliptic type} if some choice of edge lengths yields a hyperelliptic tropical curve.

Let $T$ be a tree of maximal valence 3 and fix a disjoint copy $T'$ of $T$.  Let $\iota \maps T \to T'$ be the involution that identifies each vertex of $T$ with the corresponding vertex of $T'$.  Following \cite[Definition~4.7]{chan2013tropical}, we define the \textit{ladder} over $T$ to be the graph $L(T)$ obtained by adding $3 - \val(v)$ parallel edges between $v$ and $\iota(v)$ for each $v \in V(T)$.  We call the edges added in this way \textit{vertical edges}.
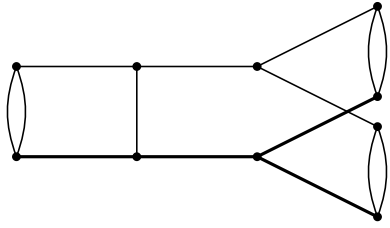
\begin{figure}[htbp]
\centering
\begin{tikzpicture}[semithick, scale=1.6, >=Stealth, decoration={markings,mark=at position 0.5 with {\arrow[xshift={2pt + 2.25\pgflinewidth}]{Latex[scale=0.8]},}}]
  \coordinate (A) at (0,0);
  \coordinate (B) at (0,-0.75);
  \filldraw (A) circle [radius=0.03] ++(1,0) circle [radius=0.03] ++(1,0) circle [radius=0.03] +(1,-0.5) circle [radius=0.03] ++(1,0.5) circle [radius=0.03];
  \filldraw (B) circle [radius=0.03] ++(1,0) circle [radius=0.03] ++(1,0) circle [radius=0.03] +(1,-0.5) circle [radius=0.03] ++(1,0.5) circle [radius=0.03];
  \draw (A) -- ++(2,0) -- +(1,-0.5) +(0,0) -- +(1,0.5);
  \draw[very thick] (B) -- ++(2,0) -- +(1,-0.5) +(0,0) -- +(1,0.5);
  \draw[bend left=20] (A) to +(B) ++(3,0.5) to +(B) ++(0,-1) to +(B);
  \draw[bend right=20] (A) to +(B) ++(3,0.5) to +(B) ++(0,-1) to +(B);
  \draw (A) ++(1,0) to +(B);
\end{tikzpicture}
\caption[]{The graph $L(T)$ for the tree $T$ marked with bold edges}
\end{figure}

\begin{lemma}\label{lem:ladder}
  Let $G$ be a $2$-edge-connected graph that is strongly of hyperelliptic type.  Then $\alpha(G) = 0$.
\end{lemma}
\begin{proof}
  We claim that $G$ is a permissible minor of $L(T)$ for some tree $T$ of maximal valence $3$.  Indeed, this follows almost immediately from \cite[Theorem~4.9]{chan2013tropical}, which states that ladders are precisely the maximal cells of the moduli space of $2$-edge-connected hyperelliptic tropical curves.  However, since we do not allow weighted vertices, we must delete loops rather than contracting them.  Then without loss of generality, we may assume that $G = L(T)$.

  We show that $\Theta(\Upsilon)$ is identically zero for $L(T)$.  Fix a 1-valent vertex $v_0$ of $T$ and an edge $e_0$ of $L(T)$ from $v_0$ to $\iota(v_0)$.  Let $S$ be the spanning tree of $L(T)$ with edges $E(T) \cup E(\iota(T)) \cup \set{e_0}$.  Each remaining vertical edge determines a unique cycle in $L(T)$; label these edges arbitrarily and orient them from $T$ to $\iota(T)$.  Orient each $e \in E(T)$ away from $v_0$; there is a partial order on $E(T)$ given by declaring that $e < e'$ whenever the unique path in $T$ from $v$ to the tail of $e'$ contains $e$.  Rephrasing \cref{prop:alpha-edge-char}, nonzero terms of $\Theta(\Upsilon)$ correspond to pairs of distinct edges $e$ and $e'$ satisfying:
\begin{enumerate}[($*$)]
\item $e$ and $e'$ share a cycle and each is part of another cycle that the other is not in.
\end{enumerate}
We claim that ($*$) is not satisfied for any pair of edges in $L(T)$.  Indeed, neither edge can be any of the vertical edges: $e_0$ is part of every cycle, while each other vertical edge is part of only one cycle.  If both $e$ and $e'$ lie in $T$, then either $e < e'$ or the two edges are incomparable.  In the first case, every cycle containing $e'$ also contains $e$.  In the second case, there are no common cycles between $e$ and $e'$.  By symmetry, ($*$) also fails if both $e$ and $e'$ lie in $T'$.  If $e' = \iota(e)$, then $e'$ and $e$ form a separating pair; in particular, they are contained in precisely the same cycles.  If $e \in T$ and $e' \in T'$ with $e' \neq \iota(e)$, then the fact that ($*$) fails for $e$ and $\iota(e')$ implies that it also fails for $e$ and $e'$.  This proves the claim, so $\Theta(\Upsilon) = 0$, as desired.
\end{proof}

\begin{proof}[Proof of \cref{thm:A}]
  Suppose first that $G$ is not of hyperelliptic type.  By \cite[Theorem~1.1]{corey2021tropical}, $G$ contains $G' \in \set{K_4,L_3}$ as a minor.  In fact, \cite[Lemma~5.10]{corey2022ceresa} implies that, in this special case, $G'$ must be a permissible minor of $G$.  That $\alpha(G) \neq 0$ follows from our computations in \cref{ex:k4,ex:l3} showing that $\alpha(G') \neq 0$ in either case.

Conversely, suppose that $G$ is of hyperelliptic type.  By \cite[Theorem~1.1]{corey2021tropical}, $G$ does not contain $K_4$ or $L_3$ as a minor.  Then neither do the maximal 2-connected components of the 2-edge-connectivization $G^2$, so any such component is still of hyperelliptic type.  In particular, \cref{lem:2,lem:3} imply that we may assume that $G$ itself is 2-connected.  By definition, there exists a 2-connected tropical curve $C$ of hyperelliptic type with underlying graph $G$.  Then by \cite[Theorem~4.5]{corey2021tropical}, there exists a hyperelliptic tropical curve $C'$ of which $C$ is a permissible minor.  Let $G'$ denote the underlying graph of $C'$.  Contracting any separating edges, the resulting tropical curve ${C'}^2$ is hyperelliptic by \cite[Corollary~3.11]{chan2013tropical}.  Then ${G'}^2$ is 2-edge-connected and strongly of hyperelliptic type, so \cref{lem:ladder} implies that $\alpha({G'}^2) = 0$.  By \cref{lem:2}, $\alpha(G') = 0$; since $G$ is a permissible minor of $G'$, we have that $\alpha(G) = 0$.
\end{proof}

\bibliographystyle{alpha}
\bibliography{references}

\end{document}